\documentclass[11pt]{amsart}
\usepackage{amssymb}
\usepackage{amsmath}
\usepackage{amsthm}
\usepackage{amsbsy}
\usepackage{psfrag}
\usepackage{pstricks}
\usepackage{graphics}
\usepackage{graphicx}
\usepackage{caption}
\usepackage{subcaption}
\theoremstyle{plain}
\newtheorem{theorem}{Theorem}[section]
\newtheorem{lemma}[theorem]{Lemma}
\newtheorem{corollary}[theorem]{Corollary}
\newtheorem{proposition}[theorem]{Proposition}

\theoremstyle{remark}
\newtheorem{remark}[theorem]{Remark}

\newcommand{\calF}{ \mathcal{F}}
\newcommand{\Th}{\mathcal{T}_h}
\newcommand{\K}{T}
\renewcommand{\O}{\Omega}
\newcommand{\Eh}{{{\mathcal E}_h}}
\newcommand{\Eho}{{{\mathcal E}^{o}_h}}
\newcommand{\Ehb}{{{\mathcal E}^{\partial}_h}}

\DeclareMathOperator\supp{supp}

\begin{document}
\allowdisplaybreaks[4]
\numberwithin{figure}{section}
\numberwithin{table}{section}
 \numberwithin{equation}{section}
%
\title[{A Posteriori Error Analysis in $L^{\infty}$-norm of DG Methods for Obstacle
Problem}]{Pointwise A Posteriori Error Analysis of a Discontinuous Galerkin Method
for the Elliptic Obstacle Problem}
\author{Blanca Ayuso de Dios} \thanks{}

\address{Dipartimento di Matematica e Applicazioni, Universit\`a di Milano Bicocca, Via Cozzi 55, 20125 Milano, Italy, \& }
\address{ Istituto di Matematica Applicata e Tecnologie Informatiche Enrico Magenes del C.N.R.,
via Ferrata, 27100, Pavia, Italy}

\email{blanca.ayuso@unimib.it}

 \author{Thirupathi Gudi}\thanks{}

\address{Department of Mathematics, Indian Institute of Science, Bangalore - 560012}
\email{gudi@math.iisc.ernet.in}

\author{Kamana Porwal}\thanks{The third author's work is supported in part by The SERB-MATRICS grant and in part by IITD SEED grant}

\address{Department of Mathematics, Indian Institute of Technology Delhi - 110016}

\email{kamana@maths.iitd.ac.in}

\date{}
\begin{abstract}
We present a posteriori error analysis in the supremum norm for the symmetric interior penalty discontinuous Galerkin method for the elliptic obstacle problem. We construct discrete barrier functions based on appropriate corrections of the conforming part of the solution that is obtained via a constrained averaging operator. The corrector function accounts properly for the non-conformity of the approximation and it is estimated by direct use of the Green's function of the unconstrained elliptic problem. The use of the continuous maximum principle guarantees the validity of the analysis without mesh restrictions but shape regularity. The proposed residual type estimators are shown to be reliable and efficient.
 Numerical results in two dimensions are included to verify the theory and validate the performance of the error estimator.
\end{abstract}
\keywords{finite element, discontinuous Galerkin, a posteriori
error estimate, obstacle problem, variational inequalities,
Lagrange multiplier, pointwise}
\subjclass{65N30, 65N15}
\maketitle
\allowdisplaybreaks
\def\R{\mathbb{R}}
\def\cA{\mathcal{A}}
\def\cK{\mathcal{K}}
\def\cN{\mathcal{N}}
\def\p{\partial}
\def\O{\Omega}
\def\bbP{\mathbb{P}}
\def\cV{\mathcal{V}}
\def\cM{\mathcal{M}}
\def\cT{\mathcal{T}}
\def\cE{\mathcal{E}}
\def\bF{\mathbb{F}}
\def\cF{\mathcal{F}}
\def\bC{\mathbb{C}}
\def\bN{\mathbb{N}}
\newcommand{\n}{{\bf n}} 
\newcommand{\nK}{{\bf n}_{\K}}
\newcommand{\dx}{\,\mbox{d}x} 
\newcommand{\ds}{\,\mbox{d}s} 
\newcommand{\jump}[1]{\lbrack\!\lbrack\,#1\,\rbrack\!\rbrack} 
\newcommand{\media}[1]{\left\{\!\!\left\{ #1 \right\}\!\!\right\}}
\newcommand{\trin}{|\!|\!|}
\newcommand{\nn}{|\!|} 
\newcommand{\re}{ \mathbb{R}}

\def\ssT{{\scriptscriptstyle T}}
\def\HT{{H^2(\O,\cT_h)}}
\def\mean#1{\left\{\hskip -5pt\left\{#1\right\}\hskip -5pt\right\}}
\def\smean#1{\{\hskip -3pt\{#1\}\hskip -3pt\}}
\def\sjump#1{[\hskip -1.5pt[#1]\hskip -1.5pt]}
\def\jumptwo{\jump{\frac{\p^2 u_h}{\p n^2}}}

\section{Introduction}\label{sec:Intro}

Several problems in a variety of fields of science (Physics, Biology, Industry, Finance, Geometry to mention a few) are often described by partial differential equations (PDEs) that exhibit a-priori unknown interfaces or boundaries. They are typically called free boundary problems and constitute an important
line of research for the mathematical and the numerical analysis communities.\\
Among these problems, the most classical one is the elliptic obstacle problem that arises in diverse applications in elasticity, fluid dynamics and operations research, to name a few\cite{Ciarlet:1978:FEM,KS:2000:VI,Glowinski:2008:VI}. It is also a prototype of the elliptic variational inequalities of the first kind. It's variational formulation reads: find $u\in \cK$ such that
\begin{equation}
\int_\Omega \nabla u\cdot\nabla (u- v)\,dx\leq (f,u-v)\quad \forall v\in \cK, \label{eq:MP}
\end{equation}
where $\Omega\subset \R^d, ~d=2,3$ is a bounded, polygonal ($d=2$) or polyhedral ($d=3$) domain with  boundary $\partial\Omega$, 
$f \in L^{\infty}(\Omega)$ denotes the force and  $(\cdot,\cdot)$ refers to the $L^{2}(\O)$ inner-product. For any $1\leq r \leq \infty $ and $D \subset \O$, we denote the $L^{r}(D)$ norm by $\|\cdot \|_{L^{r}(D)}$. The set $D^o$ denotes the interior of D for any set $D \subset \O$. We further use standard notations of Sobolev spaces \cite[Chapter 3]{adams}. {The associated norm on Sobolev space $W^{m,r}(\O)$  is denoted by $\|\cdot\|_{W^{m,r}}$ and semi-norm by $|\cdot|_{W^{m,r}}$. The Sobolev space $W^{m,r}(\O)$ with $m=0$ will denote the standard $L^{r}(\O)$ space.} The set of admissible displacements $\cK$ is a non-empty, closed and convex set
defined by
\begin{equation}\label{def:K}
\cK:=\{v\in V : v\geq \chi \;\;\text{a.e. in}\;\;
\Omega\},
\end{equation}
with $V:=H^1_0(\Omega)$  and $\chi\in H^1(\Omega)\cap C^{0,\alpha}(\overline{\Omega}), \quad 0<\alpha\leq 1$ denoting the obstacle which satisfies the compatibility condition  $\chi\leq 0$ on $\partial\Omega$. 
In \eqref{eq:MP}, the solution $u$ could be regarded as the equilibrium
position of an elastic membrane subject to the load $f$ 
whose boundary is held fixed ($u\in H^{1}_0(\O)$)  and which
is constrained to lie above the given obstacle $\chi$.
Such constraint results in non-linearity inherent to the underlying PDE, the Poisson equation here. 
{ The classical theory of Stampacchia \cite[Chapter 1, page 4]{Glowinski:2008:VI}, \cite[Chapter 2, page 40]{KS:2000:VI} guarantees the existence and uniqueness of the solution.
As for the regularity of the solution, it is shown in~\cite{Frehse:1978:Regularity,CK:1980:Regularity1}, that under our assumptions on the domain and the data, $f \in L^{\infty}(\Omega)$ and $\chi\in H^1(\Omega)\cap C^{0,\alpha}(\overline{\Omega}), ~~ 0<\alpha\leq 1$,  the solution to the variational inequality \eqref{eq:MP} is also H\"older continuous i.e. we have $u \in H_0^1(\Omega)\cap C^{0,\beta}(\overline{\Omega})$ for $0<\beta\leq \alpha$.

 }


\par \noindent
The contact (or coincidence) and non-contact sets of the exact solution $u$ are defined as
\begin{equation}\label{cnc}
\begin{aligned}
\mathbb{C}&:= \left\{x \in \O: u(x)=\chi(x)\right\}^{o}, &&\\
\mathbb{N}&:=\left\{x \in \O: u(x)>\chi(x)\right\}. &&
\end{aligned}
\end{equation}
The free boundary here is the interior boundary of the coincidence set $ \mathbb{C}$. The H\"older continuity of $u$ guarantees that both $\mathbb{C}$ and the free boundary are closed sets.
\par \noindent
{ From Riesz representation theorem, associated to the solution $u$ we introduce the linear functional $\sigma(u)\in H^{-1}(\Omega)$ by }

\begin{align}\label{eq:sigmadef}
\langle \sigma(u),v\rangle =(f,v)-a(u,v),\quad \forall v\in
H^1_0(\Omega),
\end{align}
where $\langle \cdot,\cdot\rangle$ refers to the duality pairing of
$H^{-1}(\Omega)$ and $H^1_0(\Omega)$. 
From \eqref{eq:sigmadef} and \eqref{eq:MP}, it follows that
\begin{equation}\label{eq:sigma}
\langle \sigma(u),v-u\rangle\leq 0,\quad  \forall v \in \cK.
\end{equation}
In particular\footnote{ Choosing  $0 \leq \phi  \in H_0^1(\Omega)$ and setting $v=u+\phi $ in \eqref{eq:sigma}, we get $\sigma(u) \leq 0$. Further, if $u> \chi$ on some open set $D \subset \Omega$ then $\sigma(u)= 0$ on $D$ (see \cite[Chapter 2, page 44]{KS:2000:VI}) }, $\sigma(u)=0$ on the non-contact set $\mathbb{N}$ { and the support of $\sigma(u)$ is contained in the contact set $\mathbb{C}$. In fact, \cite[Chapter II, Section 6]{KS:2000:VI} the Riesz representation theorem states that  there
 is a finite non-negative radon measure $-d \sigma$ supported on $\mathbb{C}$ satisfying:
$$
\langle \sigma(u),v\rangle = \int_{\O} v d\sigma= (f,v)-a(u,v) \qquad \forall\, v \in C_{0}(\bar\Omega),
$$
where $C_{0}(\bar\Omega)$ denotes the space of continuous functions on $\bar\O$ with compact support. Often, $\sigma(u)$ is also referred as the continuous Lagrange multiplier.  We wish to stress  that the solution operator of \eqref{eq:MP}  is not only non-linear and non-differentiable, but also strikingly not one-to-one in the sense that any variation in $f$ within the contact set might or might not result in a variation in the solution $u$.}


\par \noindent
A-priori error analysis in the energy norm for \eqref{eq:MP} are discussed in  \cite{Falk:1974:VI,BHR:1977:VI,BHR:1978:Mixed, Wang:2002:P2VI,GG:2018:3DObst} for conforming methods and in \cite{WHC:2010:DGVI} for DG methods. In energy norm, several  a-posteriori error estimators  are studied in ~\cite{BC:2004:VI,CN:2000:VI,Braess:2005:VI,Veeser:2001:VI,NPZ:2010:VI,TK:2015:VIP2, GG:2018:InObst} for conforming approximation of variational inequalities while residual type a posteriori error estimators have been derived in \cite{TG:2014:VIDG,TG:2014:VIDG1} for DG approximation. 
In this article, we study  a-posteriori error analysis in the supremum norm, for the symmetric interior penalty discontinuous Galerkin (SIPG) approximation of the elliptic obstacle problem \eqref{eq:MP}. The rationality of our choice of norm comes from the fact that estimates in the maximum norm are relevant for variational inequalities since they  provide further localized information  on the approximation. Moreover, they also constitute a cornerstone for a future subsequent convergence analysis of the free boundary \cite{BrezziCaffarelli, NochettoFB}.\\
\par \noindent
A-priori error analysis in $L^{\infty}$ norm goes back to the seminal works by Baiocchi and Nitsche \cite{Baiocchi:1977:VI,Nitsche:1977:VI}, for conforming finite linear finite elements. As the proof uses a discrete maximum principle, the analysis requires the classical assumption on acute angles for the mesh partitions \cite{Ciarlet:1978:FEM}.
Unlike for a-priori analysis,  in \cite{NSV:2002:Ptwise,NSV:FullyLoc} the authors accomplish the $L^{\infty}$ norm a-posteriori analysis of the obstacle problem, by using a {\it dual approach} to Baiocchi \cite{Baiocchi:1977:VI} and making use of the {\it continuous} maximum principle, hence avoiding any mesh restriction except the standard shape regularity.
 In \cite{NSV:2002:Ptwise}, a positivity preserving interpolation operator is used and in \cite{NSV:FullyLoc} the estimators and the analysis is refined to ensure further localization of the estimators.
 The ideas in \cite{NSV:2002:Ptwise,NSV:FullyLoc} have been extended in \cite{NSV:2006} to the conforming approximation of monotone semi-linear equations and in \cite{FV:2003:Ptwise4}, for  the conforming approximation of a double obstacle problem arising from regularized minimization of a related functional, providing a uniform error estimator  with respect to the regularization parameter that can be further used in the minimization of the discrete functional. { We also refer to the nice survey \cite{NochettoSalgado} for the approximation of classical and fractional obstacle problems.\\ }

\par \noindent
In this work, we follow the path laid down in \cite{NSV:2002:Ptwise,NSV:FullyLoc}, extending it and adapting it for the SIPG method. We stress that the nonconformity of the approximation and the nonlinearity of the problem, preclude the extension from being trivial. Also, the discontinuous nature of the finite element space allow for a slightly better localization of the error estimators.
To be able to use the continuous maximum principle, which permits to avoid mesh restrictions, we construct discrete barrier functions for the solution $u$ that belong to $H^{1}$. This in turn,  forces us to construct the discrete sub and super-solutions starting from the conforming part of the DG approximation and amend it with some corrector function that account for the nonconformity of our approximation. The conforming part of the DG approximation is obtained by using an averaging operator introduced in \cite{TG:2014:VIDG} that ensures the obstacle constraint are preserved. 
The proof of the reliability of the estimator relies then on providing a pointwise estimate on the corrector function. Here, our approach varies from  \cite{NSV:2002:Ptwise,NSV:FullyLoc} where a regularized Greens' function (for the unconstrained elliptic problem) is employed. Instead, we consider the direct use of Greens' functions catering for the sharpened estimates (on the Greens' function)  proved in \cite{DG:2012:Ptwise2, GG}. As a result, the final estimators are improved in the exponent of the logarithmic factor with respect to \cite{NSV:2002:Ptwise}, but more significantly, this approach allows for accounting  appropriately the non-conformity in the approximation and  evidences that the analysis performed for SIPG method, cannot be extended to any of the non-symmetric methods of the interior penalty family. The local efficiency estimates of the error estimator are obtained  using the bubble functions techniques \cite{AO:2000:Book}. 
{ To prove the efficiency of the estimator term that measures the error in approximation of the obstacle in the discrete contact set and discrete free boundary set, unlike in  \cite{NSV:2002:Ptwise,NSV:FullyLoc} we need to assume $\chi<0$ on the part of $\partial \O$ that intersects the free boundary. This assumption is not very restrictive and is done in many research works. We accept this technical restriction in order to keep the locality of the definition of the discrete Lagrange multiplier granted by the DG construction, which is computationally simpler than the ones introduced in \cite{NSV:2002:Ptwise,NSV:FullyLoc}.  }
Numerical experiments in two dimensions confirming the applicability of a posteriori estimates within adaptive schemes are presented.\\
\par
\noindent 
{
We expect that the analysis presented could be helpful in understanding the convergence of the DG approximation to the free boundary, taking into account the seminal works~\cite{BrezziCaffarelli, NochettoFB}. This would allow to assess if the often acclaimed flexibility of  DG methods could bring any real benefit with respect to the conforming approximation in this context. This will be the subject of future research. The ideas developed here might be used in the analysis of DG approximation of control problems as well as the a-posteriori error analysis of DG methods  for other variational inequalities e.g. Signorini problem \cite{Glowinski:2008:VI}.}
{The error estimator obtained in this article improves the error estimator
obtained in \cite{NSV:2002:Ptwise} by a logarithm factor when applied
to the conforming finite element methods. Except for this logarithm term,
our error estimator is comparable to that of \cite{NSV:2002:Ptwise}.}\\
\par
\noindent 
The  article is organized as follows.
 In the next section we introduce basic notations and some auxiliary results that will be used in further analysis. DG method for approximating the elliptic obstacle problem \eqref{eq:MP} and its main properties are discussed in Section \ref{sec:dgm}. In  Section \ref{sec:Reliaibility},  we introduce the error estimators which are the main tools for the analysis and prove the major result of the paper, namely the reliability of the estimator.  The efficiency of a posteriori error estimator is discussed in Section \ref{Sec:eff}. Numerical experiments are presented illustrating the theoretical findings and validating the  performance of the estimator  in Section \ref{sec:Numerics}.\\

\par

\section{Notation and Preliminary Results}\label{sec:Prelims}
\par \noindent
In this section, we introduce useful notations and some auxiliary results which will be used throughout this work. 

\begin{align*}
\cT_h &=\text{a simplicial triangulations of } \Omega, \text{that are assumed to be shape regular}\\
T&=\text{an element of } \cT_h, \qquad |T|= \text{volume of the simplex} ~T\\
 h_T &=\text{diameter of } T, \qquad h=\max\{h_T : T\in\cT_h\}, \qquad h_{min}=\min\{h_T : T\in\cT_h\}\\
\cV_h^i&=\text{set of all vertices in } \cT_h \text{ that are in }
\O\\
\cV_h^b&=\text{set of all vertices in } \cT_h \text{ that are on }
\p\O\\
\cV_h&=\cV_h^i\cup\cV_h^b\\
\cV_T&=\text{set of three vertices of simplex}~ T\\
\Eho &=\text{set of all interior edges/faces of } \cT_h\\
\Ehb &=\text{set of all boundary edges/faces of } \cT_h\\
\cE_h &=\Eho\cup\Ehb \\
h_e &=\text{diameter of  any edge/face } e\in\cE_h\\
\cT_p&=\text{set of all elements sharing the vertex } p\\
\cT_e&=\text{patch of two elements sharing the face } e\in \Eho ~\text{i.e.} ~\cT_e=\{T_+,T_-\}, \text{where} ~T_+ ~\text{and} ~T_- ~\text{are two} \\ & \quad \text{elements sharing the face}~ e\\
\cE_p&=\text{set of all edges/faces connected to the vertex } p\\
\nabla_h &=\text{piecewise (element-wise) gradient}.
\end{align*}


\par \noindent
 For any $T \in \Th,~ p\in \cV_h$ and $e \in \Eho $, we define,
\begin{align}
\omega_{T}  & := \cup \left\{ T' \in \Th \,\, :\,\, T'\cap T\ne \emptyset \right\} &&  \label{omega0} \\
\omega_{p} & :=\cup \{  T' \in \Th \,\, : T' \in \cT_p\}  && \label{omegaP}\\
\omega_e & :=T_+\cup
T_- \qquad. \label{omegaE}&&
\end{align}

\par \noindent

Further,
\begin{itemize}
\item For any function $v$, $v^{+}$ denotes it's positive part i.e. $v^{+}(x)=\max\{v(x),0\}$
\item For any function $v$, $supp(v)$ denotes support of the function $v$
\item $\|\cdot\|_{L^{\infty}(\cE_h)} := \displaystyle\max_{e \in \cE_h}\|\cdot \|_{L^{\infty}(e)}$
\item $\|\cdot\|_{L^{\infty}(\cT_h)}:= \displaystyle\max_{T \in \cT_h}\|\cdot \|_{L^{\infty}(T)}$
\item $X \lesssim Y$: there exist a positive constant $C$, independent of mesh parameter and the solution such that $X \leq C Y$.  
\end{itemize}
\par
\noindent
The shape regularity  assumption on $\cT_h$ 
\cite[Chapter 3, page 124]{Ciarlet:1978:FEM} implies that 
$$|\cE_p|  \lesssim   1 \quad \forall p\in \cV_h,$$
where $|\cE_p|$ denotes the cardinality of ${\cE_p}$.
\par
\noindent
In order to define the jump and mean of discontinuous
functions appropriately, we introduce the broken Sobolev space
\begin{eqnarray*}
H^1(\O,\cT_h)=\{v\in L^{2}(\Omega) :\,v_{\ssT}= v|_{T}\in
H^1(T)\quad\forall\,~T\in{\mathcal T}_h\}.
\end{eqnarray*}
\par \noindent
{\sc Trace Operators:}  Let
$e\in\Eho$ be shared by two elements $T_+$ and $T_-$ such that
 $e=\partial T_+\cap\partial T_-$.  If $n_+$ is the unit normal of $e$ pointing from $T_+$ to $T_-$ then $n_-=-n_+$.
We set $v_\pm=v\big|_{T_\pm}$ and define the jump and the mean of $v\in H^1(\Omega,{\mathcal T}_h)$
on $e$  by
\begin{eqnarray*}
 \sjump{v} = v_+n_++v_-n_- \,\quad \mbox{and}\,\quad \smean{v} = \frac{1}{2}(v_++v_-),\quad \mbox{respectively.}
\end{eqnarray*}

\par \noindent
Similarly for a vector valued function $q\in H^1(\Omega,{\mathcal
T}_h)^d$, the jump and mean on $e\in\Eho$ are defined by
\begin{eqnarray*}
 \sjump{q} = q_+\cdot n_++q_-\cdot n_-\quad\,\mbox{and}\quad\,\smean{q} = \frac{1}{2}(q_++q_-).
\end{eqnarray*}
\par \noindent
 For a boundary face $e\in \Ehb$, with  $n_e$ denoting the unit normal of $e$ that points outward to $T$ for $\partial T\cap \partial\Omega=e$, we set: 
\begin{align*}
 \sjump{v}= vn_e\,\quad \mbox{and}\,\quad \smean{v}=v, \quad \forall\, v\in H^1(T), &&\\
 \sjump{q}= w\cdot n_e,\,\,\mbox{and}\,\,\smean{q}=q,\quad \forall\, q\in H^1(T)^d. &&\
\end{align*}
Throughout the article, $C$ denotes a generic positive constant which is independent of the mesh size $h$.\\
{\sc Discrete Spaces:} Let $\mathbb{P}^{k}(T)$, $k \geq 0$ integer, is the space of polynomials defined on $T$ of degree less
than or equal to $k$. We define the first order discontinuous finite element space as
\begin{eqnarray*}
V_h=\{v\in L^{2}(\Omega) : v|_T\in\mathbb{P}^{1}(T),\,T\in \cT_h\}.
\end{eqnarray*}
The conforming finite element space, denoted here by $V_h^{conf}$ is then defined as  
$$
V_h^{conf}
=V_h\cap H_0^1(\Omega).
$$
 Let $\chi_h \in V_h^{conf}$ be the nodal
interpolation of $\chi$. We define the discrete analogue of
$\cK$ by
\begin{equation}\label{kh}
\cK_h:=\{v_h\in V_h: v_h|_T(p)\geq \chi_h(p),\;\; \forall p\in \cV_T,\;\; \forall T\in\cT_h\},
\end{equation}
\noindent
which is a nonempty, closed and convex subset of $V_h$. {Note however that, since $\chi$ is a general function, we have $\cK_h \nsubseteq \cK$. In case of $\chi$ being an affine function, we have $\cK_h \cap H_0^1(\Omega) \subseteq \cK$.}

\par
\subsection{Some Discrete Operators}\label{sec:Smoothing}
We now introduce some discrete operators that will be used in further analysis.\\
Following  \cite{DG:2012:Ptwise2},  we define a local projection operator $\Pi_h:L^{1}(\Omega)\longrightarrow V_h$.
For any $T \in \cT_h$, let $(\phi_p^T)_{p \in \mathcal{V}_T}$  be associated Lagrange basis functions. Then, for any $v\in L^{1}(T)$, we define $\Pi_h^T: L^{1}(T)\longrightarrow V_h$ by
\begin{equation}\label{defPik}
 \Pi_h^T v(x):=\sum_{p \in \mathcal{V}_T}\phi_p^T(x)\int_{T} \psi_p^T(s) v(s)~ds,
\end{equation}
where
$(\psi_p^T)_{p \in \mathcal{V}_T}$ is the $L^{2}(T)$ dual basis of $(\phi_p^T)_{p \in \mathcal{V}_T}$ i.e.
\begin{align*}
\int_{T} \phi_p^T(x) \psi_q^T (x)~dx= \delta_{pq}, \quad p,q \in \mathcal{V}_T.
\end{align*}
Then,  $\Pi_h$ is defined by setting  $\Pi_h v|_T  =\Pi_h^T v$ for each $T\in \Th$.
Observe that from \eqref{defPik} we do have $\Pi_h^T\phi_p^T=\phi_p^T$ which implies $\Pi_h^T$ is indeed a projection, and so is $\Pi_h$. Hence we have $\Pi_h v_h=v_h$ for all $v_h\in V_h$.

\noindent
The following Lemma from \cite{DG:2012:Ptwise2} provides stability and approximation properties of $\Pi^{T}_h$. It's proof follows by using  Bramble Hilbert lemma  and standard scaling arguments; see  \cite{DG:2012:Ptwise2} for details.
\begin{lemma}\label{lem:approx_proj}
Let $T\in \Th$  and let  $\Pi_h^T: L^{1}(T)\longrightarrow V_h$ be the local projection defined in \eqref{defPik}. Let $s$ be an integer such that $ 0\leq s\leq 2$ and  $\psi\in W^{s,1}(T)$.  Then,
\begin{align}
|\Pi^T \psi|_{W^{s,1}(T)} & \lesssim  | \psi|_{W^{s,1}(T)},  \qquad  0\leq s\leq 2 &&\label{stabPik0} \\
\sum_{k=0}^{s} h_{T}^{k-s} |\psi-\Pi^T_h\psi|_{W^{k,1}(T)} & \lesssim | \psi|_{W^{s,1}(T)}, \qquad  0\leq s\leq 2. &&\label{stabPik2}
\end{align}

\end{lemma}

\noindent We now revise the following inverse and trace inequalities which will be frequently used in our later analysis.\\
{\it Inverse Inequalities~\cite[Chapter 3, page 140]{Ciarlet:1978:FEM}:} For any $v_h\in V_h$ and $1 \leq p\,, q  \leq \infty$,
\begin{align}
\|v_h\|_{W^{m,p}(T)} &  \lesssim h_{T}^{\ell -m}   h_{T}^{d\left(\frac{1}{p}-\frac{1}{q}\right)}
\|v_h\|_{W^{\ell,q}(T)}\quad &\forall\,T\in \cT_h,\label{eq:inverse}\\
\|\nabla v_h\|_{L^{p}(T)} &  \lesssim h_T^{-1}
\|v_h\|_{L^{p}(T)}\quad&\forall\,T\in \cT_h,\label{eq:inverse1}\\
\|v_h\|_{L^{\infty}(T)} &  \lesssim  h_T^{-\frac{d}{2}}
\|v_h\|_{L^{2}(T)}\quad &\forall\,T\in \cT_h,\label{eq:inverse2} \\
\|v_h\|_{L^{\infty}(e)} & \lesssim h_e^{\frac{1-d}{2}}\|v_h\|_{L^{2}(e)} \quad &\forall\,e\in \Eh,\label{eq:inverse3}
\end{align}
where $h_e$ and $h_T$ denote, respectively,  the diameters of the face $e$ and the element $T$.

{\it Trace inequality~\cite{agmon65}}: Let $\psi\in W^{1,p}(T), ~~~T\in \cT_h$
and let $e\in\cE_h$ be an edge/face of $T$. Then for any $1 \leq p < \infty$, it holds that
{
\begin{equation}\label{eq:traceAg}
\|\psi\|_{L^{p}(e)}^p \lesssim h_e^{-1} \big( \|\psi\|_{L^{p}(T)}^p+
h_e^{p}\|\nabla \psi\|_{L^{p}(T)}^p\big),
\end{equation}
}
where $h_e$ denotes the diameter of the face $e$.

{\sc Averaging or Enriching Operator:} We close the section by introducing an {averaging} (sometimes called enriching) operator $E_h:V_h\longrightarrow
V_h^{conf}$  that  plays an essential role in the analysis.
Such type of operator was first introduced in the seminal work \cite{KP:2003:APosterioriDG}, where the first a posteriori error analysis for DG approximations of elliptic problems was presented. For the obstacle problem, it is necessary to ensure that $E_h$ exhibits an
additional constraint preserving  property, as proposed in \cite{TG:2014:VIDG}. This construction will be essential in the forthcoming analysis.
\par
\noindent
Bearing in mind that any function in linear conforming finite element space $V_h^{conf}$ is uniquely
determined by its nodal values at the vertices $\cV_h$ of $\cT_h$, 
for $v_h\in V_h$, we define $E_hv_h\in V_h^{conf}$ as follows:
\begin{equation}\label{eq:Enrich}
E_hv_h=\sum_{p\in\cV_h} \alpha(p) \phi^{h}_{p}(x) \mbox{     with     } \alpha(p) := \left\{ \begin{array}{ll} 0 & \text{ if } p\in\cV_h^b,\\\\
\min\{v_h|_T(p):T\in\cT_p\} & \text{ if } p\in\cV_h^i.
\end{array}\right.
\end{equation}
Observe that from its definition it follows immediately that
\begin{equation}\label{eq:Ehi2}
E_hv_h(p)\geq \chi_h(p)\quad \forall v_h\in \cK_h,\quad \forall
p\in\cV_h,
\end{equation}
and in particular $E_hv_h\in \cK_h\cap V_h^{conf}$ for all $v_h\in\cK_h$.

\par
\smallskip \noindent
Next Lemma provides some approximation properties for the operator $E_h$. Similar estimates for general polynomial degree for the unconstrained version of the averaging operator introduced in \cite{KP:2003:APosterioriDG} can be found in \cite{DG:2012:Ptwise2}.
\begin{lemma}\label{lem:Eapen4}
Let $E_h:V_h\longrightarrow
V_h^{conf}$ be the operator defined through \eqref{eq:Enrich}. Then, for any $v\in V_h$ it holds
\begin{align}
\|E_hv-v\|_{L^{\infty}(\cT_h)} & \lesssim  \| \sjump{v}\|_{L^{\infty}(\cE_h)} \label{eq:apen3} \\
\text{and}\qquad \qquad \qquad \qquad \qquad& \notag\\
\max_{T \in \cT_h} ~ h_T\|\nabla(E_hv-v)\|_{L^{\infty}(T)} & \lesssim  \| \sjump{v}\|_{L^{\infty}(\cE_h)}. \label{eq:apen4}
\end{align}
\end{lemma}

\begin{proof}

Let $T\in\cT_h$. Note that the Lagrange basis functions $\phi_p^T$ satisfies the property~ $\|\phi_p^T \|_{L^{\infty}(T)}= 1~ \forall ~ p \in \cV_h$. For any $p\in \cV_T\cap \cV_h^i$, using the definition of $E_h$ we have the existence of
some $T_*\in \cT_p$ such that $E_hv(p)=\alpha(p)=v|_{T_*}(p)$ and $E_hv(p)=0$ for $p\in\cV_T\cap \cV^b_h$. Therefore, for any $v\in V_h$, we have,
\begin{align*}
\|E_hv-v\|_{L^{\infty}(T)} &= \|\sum_{p \in \cV_T}(\alpha(p)-v(p)) \phi_p^T\|_{L^{\infty}(T)} \\
& \leq   \sum_{p \in \cV_T\cap\cV_h^i} |\alpha(p)-v(p)|+\sum_{p \in \cV_T\cap\cV_h^b} |\alpha(p)-v(p)| \\
& \leq   \sum_{p \in \cV_T\cap\cV_h^i} |v|_{T_*}(p)-v(p)|+\sum_{p \in \cV_T\cap\cV_h^b} |v(p)| \\
&{\lesssim  \sum_{p \in \cV_T} \| \sjump{v}\|_{L^{\infty}(\cE_p)}, }
\end{align*}
which leads to \eqref{eq:apen3}. Finally, the estimate \eqref{eq:apen4} follows from the inverse inequality \eqref{eq:inverse1}.

\end{proof}

\subsection{Some results on Green's functions and Regularity for the unconstrained problem}
{ 
To prove the reliability estimates we will make use of the Green's function for an unconstrained Poisson problem. We collect here several results that will be further required in our analysis.
We start with an elementary Lemma on the regularity of the solution of the unconstrained problem.

\begin{lemma}\label{lem:ax0A}
Let $g\in W^{-1,r}(\Omega)$ for some $r>d\geq 2$. Then, there exists a unique weak solution $y \in W^{1,r}(\Omega) \cap H^{1}_0(\Omega)$ to the problem
\begin{equation}\label{pois00}
-\Delta y =g \quad \mbox{ in   } \Omega\;, \qquad \qquad y=0 \mbox{    on     } \partial \Omega\;,
\end{equation}
such that it satisfies the estimate
\begin{equation}\label{eq:21}
\|y\|_{W^{1,r}(\Omega)}\leq C_r \|g\|_{W^{-1,r}(\Omega)},
\end{equation}
with $C_r=Cr$.
Furthermore, $y \in C^{0}(\Omega)$ by virtue of the Sobolev embedding.
\end{lemma}
\par \noindent
The proof of estimate \eqref{eq:21} is done using the Calderon-Zygmund theory of  of singular integral operators (see \cite{Gilbard-Trudingerbook}; the estimate holds for $r\ge 2$). The continuity of the solution is guaranteed by Sobolev embedding since $y \in W^{1,r}(\Omega),\,\, r>d$. We remark that, although not advocated here,  a-priori bounds on the $L^{\infty}$-norm of the solution to \eqref{pois00}  in terms of the $W^{-1,r}(\Omega)$-norm of $g$ are given in \cite[Theorem 4.2 (a)]{Stampacchia} and \cite[Lemma 1]{CASAS}\footnote{This proof, although valid for polyhedral domains, uses elliptic regularity of \eqref{pois00} which requires convexity of the domain and hence we do not use this result.}.\\
\par \noindent
The results in \cite[Section 9]{Stampacchia} guarantee that we can define a Green's function associated to problem \eqref{pois00}. More precisely,
for all $x\in \Omega$ there exists a Green's function $\mathcal{G}(x,\xi): \Omega \times \Omega \longrightarrow \mathbb{R}$, defined as the solution (in the sense of distributions) to
\begin{equation}\label{Ga0}
\left\{\begin{aligned}
-\Delta_{\xi} \mathcal{G} &= \delta_{x}(\xi), \qquad &\xi \in \Omega, &&\\
\mathcal{G}(x,\xi) &=0\qquad &\xi \in \partial \Omega, &&
\end{aligned}\right.
\end{equation}
such that for any $y\in H^{1}_{0}(\O)\cap W^{1,q}(\Omega), q>d,$ the following representation holds
\begin{equation}\label{key0}
y(x)=(\nabla y, \nabla \mathcal{G}(x,\cdot))_{L^{2}(\Omega)}\;.
\end{equation}
In \eqref{Ga0}, $\delta_{x}(\xi)=\delta(x-\xi)$ denotes the delta function around the point $x$.
  Notice that the Green's function $\mathcal{G}$  in \eqref{Ga0} has indeed a singularity at $x$. Next proposition provides some estimates on the Green's function $\mathcal{G}$ which account for the different regularity of $\mathcal{G}$ near and far from $x$.  For their proof, we refer to \cite{GG} which heavily relies on the estimates shown in \cite{DG:2012:Ptwise2}, improving them in a logarithmic factor. Both works make use of the fine estimates shown in \cite{HKin,G26,G24}.
  \begin{proposition}\label{Greens0}
  Let $\mathcal{G}$ be the Green's function defined in \eqref{Ga0}. Then for any $x\in \Omega$, the following estimate hold
  \begin{equation}\label{dem:22}
  \|\mathcal{G}(x,\cdot)\|_{L^{1}(\Omega)} + \|\mathcal{G}(x,\cdot)\|_{L^{d/(d-1)}(\Omega)} +|\mathcal{G}(x,\cdot)|_{W^{1,1}(\Omega)} \lesssim 1.
  \end{equation}
  In addition, for the ball $B(x,\rho)$ of radius $\rho$ centered at $x\in \Omega$, the following estimates hold true:
  \begin{align}
   \|\mathcal{G}(x,\cdot)\|_{L^{1}(B(x,\rho)\cap\Omega)} &\lesssim \rho^{2} \log(2+\rho^{-1})^{\kappa_d} \;, \label{24a}\\
  |\mathcal{G}(x,\cdot)|_{W^{1,1}(B(x,\rho)\cap \Omega)} &\lesssim \rho\;, \label{24d}\\
    |\mathcal{G}(x,\cdot)|_{W^{2,1}(\Omega\smallsetminus B(x,\rho)\cap \Omega)} &\lesssim \log(2+\rho^{-1})\;, \label{24e}\\
 |\mathcal{G}(x,\cdot)|_{W^{1,d/(d-1)}(\Omega\smallsetminus B(x,\rho)\cap \Omega)} &\lesssim \log(2+\rho^{-1})\;, \label{24c}
 \end{align}
 where $\kappa_{2}=1$ and $\kappa_{3}=0.$
 \end{proposition}

}

\section{Interior Penalty Discontinuous Galerkin approximation}\label{sec:dgm}
In this Section, we introduce the interior penalty (IP) discontinuous Galerkin method for approximating the obstacle problem \eqref{eq:MP}, discuss it's properties and comment on the possibility of using any of its non-symmetric versions.
\par\noindent
The IP method reads:  find $u_h \in \cK_h$ such
that
\begin{equation}
\cA_h(u_h,u_h-v_h)\leq (f,u_h-v_h)\qquad\forall\,v_h\in\cK_h,
\label{eq:FEM}
\end{equation}
where  the bilinear form $\cA_h(\cdot,\cdot)$ can be written as the sum of bilinear forms $a_h(\cdot,\cdot)$ and $b_h(\cdot,\cdot)$ i.e.
\begin{equation}
\cA_h(v_h,w_h)=a_h(v_h,w_h)+b_h(v_h,w_h) \label{eq:Ah}\quad \forall~ v_h, w_h \in V_h,
\end{equation}
with
\begin{equation}
a_h(v_h,w_h)=(\nabla_h v_h,\nabla_h w_h)\label{eq:ah},
\end{equation}
and the bilinear form $b_h(\cdot,\cdot)$ gathers the consistency and the
stability terms:
\begin{align}\label{bh}
b_h(v_h,w_h)&=-\sum_{e\in\cE_h} \int_e \big(\smean{\nabla v_h}\sjump{w_h}+ 
\smean{\nabla w_h}\sjump{v_h}\big)ds +\sum_{e\in\cE_h} \int_e
\frac{\gamma}{h_e}\sjump{v_h}\sjump{w_h}\,ds,
\end{align}
where $\gamma\geq \gamma_0>0$ with $\gamma_0$ as a sufficiently large positive number to ensure coercivity of $\cA_h(\cdot,\cdot)$.
It can be checked, using standard theory \cite{WHC:2010:DGVI}, that the discrete problem \eqref{eq:FEM} has a unique solution  $u_h \in \cK_h$.\\
\par\noindent

We wish to notice that we focus on the symmetric interior penalty method (SIPG). 
While it is typical in the DG framework to consider the whole IP family, considering both the symmetric and the two nonsymmetric methods, our pointwise a-posteriori error analysis holds only for the symmetric one. The underlying reason being, as it will be exhibited later on, the use of duality in the error analysis which entails symmetry of the method. This is highlighted in Remark \ref{rem:NONSYM} in Section \ref{sec:Reliaibility}.\\

\par\noindent
For our analysis, we need to consider an extension of $\cA_h(\cdot,\cdot)$ that allows for testing non-discrete functions less regular than $H^{1}(\O)$. Let $p> d$ and $1\leq q < \frac{d}{d-1}$ be the conjugate of $p$ i.e. $q=\frac{p}{p-1}$. Set,
\begin{equation}\label{def:M}
 \mathcal{M}=W^{1,q}_0(\O)+V_h. 
 \end{equation}
 Then, we define $\widetilde{\mathcal{A}}_h : \left(W^{1,p}_0(\O)+V_h\right) \times \left( W^{1,q}_0(\O)+V_h\right) \longrightarrow \mathbb{R}$ by
\begin{equation}\label{ext:A}
\widetilde{\cA}_h(v,w):= a_h(v,w)+\tilde{b}_h(v,w)  , \quad v \in W^{1,p}_0(\O)+V_h, ~ w \in \mathcal{M},
\end{equation}
where
\begin{align}\label{ext:b}
\tilde{b}_h(v,w)=-\sum_{e\in\cE_h} \int_e \smean{\nabla \Pi_h(v)}\sjump{w} ds   - \sum_{e\in\cE_h} \int_e \smean{\nabla \Pi_h w}\sjump{v}\,ds +\sum_{e\in\cE_h} \int_e
\frac{\gamma}{h_e}\sjump{v}\sjump{w}\,ds. 
\end{align}
Notice that since $\Pi_h$ is a projection, $\widetilde{\cA}_h(v,w)={\cA}_h(v,w)$ for $v, w \in V_h$.
Following, \cite{Veeser:2001:VI,TG:2014:VIDG}, we define the discrete Lagrange multiplier $\sigma_h\in V_h$, which could be thought of as an approximation to the functional $\sigma(u)$ defined in \eqref{eq:sigmadef}: 
\begin{equation}\label{def-sigmah}
\langle \sigma_h,v_h\rangle_h :=(f,v_h)-\cA_h(u_h,v_h) \quad \forall v_h \in V_h,
\end{equation}
\text{where} $\langle\cdot,\cdot\rangle_h$ is given by: for $w_h,v_h \in V_h,$
\begin{align}\label{sigmah0}
\langle w_h,v_h\rangle_h&= \sum_{T \in \mathcal{T}_h}\int_T {I_h}(w_h|_T v_h|_T)~dx=\sum_{T \in \mathcal{T}_h}\sum_{p \in \mathcal{V}_T}w_h(p)v_h(p)\int_T \phi_p^h~dx \nonumber \\
& =\sum_{T \in \mathcal{T}_h}\frac{|T|}{d+1} \sum_{p \in \mathcal{V}_T}w_h(p)v_h(p) ,
\end{align}
with $\{\phi_p^h\}_p$ as nodal basis functions and {$I_h$ is defined by setting  $I_h v|_T  =I_h^T v$ for each $T\in \Th$ where $I_h^T$ denotes the standard Lagrange nodal interpolation operator}. The use of the $\langle\cdot,\cdot \rangle_h$ inner product in the definition  \eqref{def-sigmah} of $\sigma_h$ allows for localizing $\sigma_h$ at the vertices of the partition, which allow for an easy implementation and  will play a crucial role in the further analysis. The use of \eqref{def-sigmah}-\eqref{sigmah0} was brought into play in \cite{Veeser:2001:VI,TG:2014:VIDG} for the a posteriori error analysis in the energy norm of conforming and DG approximations of the obstacle problem. While more refined definitions of the discrete Lagrange multiplier have been proposed in \cite{NSV:FullyLoc}  in the conforming framework, in this work we stick to this definition. The  main reason is that here, already the nature of the discontinuous spaces  grants us the possibility of localizing the discrete multiplier  $\sigma_h$ to its support; outside the discrete noncontact set.  Also, as we shall show, the use of \eqref{sigmah0}  allows for a first clean analysis of the DG approximation.\\
We define the discrete contact, non-contact and free boundary sets relative to the discrete solution $u_h$, as:
\begin{align}
\mathbb{C}_h& :=\{ T \in \cT_h : u_h(p)=\chi_h(p)~ \forall ~p \in \cV_T\} &&\label{conth}\\
\mathbb{N}_h&:=\{ T \in \cT_h : u_h(p)>\chi_h(p)~ \forall ~p \in \cV_T\} &&\label{nconth}\\
\mathbb{M}_h&= \cT_h \setminus (\mathbb{C}_h \cup \mathbb{N}_h).&&\label{freebh}
\end{align}
Basically, elements in $\mathbb{M}_h$ have the property of having of at least a vertex $p$ with $u_h(p)=\chi_h(p)$ and at least a vertex $p'$ with $u_h(p')>\chi_h(p')$, forming the discrete free boundary set.\\
\par \noindent
Using \eqref{def-sigmah} and the discrete problem \eqref{eq:Ah}, we obtain that
\begin{align}\label{eq:disRes}
\langle \sigma_h,v_h-u_h\rangle_h \leq 0 \quad \forall~ v_h\in \cK_h.
\end{align}
  Let now, $\phi^h_p \in V_h$ be the Lagrange basis function associated with the vertex $p$ (hence, it takes the value one at the node $p$ and vanishes at all other nodes). By setting $v_h=u_h+\phi^h_p$ in \eqref{eq:disRes} we obtain $\sigma_h(p)\leq 0 ~\forall~ p \in \mathcal{V}_T$.  Consequently, since the choice of the vertex $p$ is arbitrary and $\sigma_h$ is linear on $T$, we conclude $\sigma_h\leq 0 ~\text{in}~ T ~~ \forall~ T \in \mathcal{T}_h$.
Next, let $p$ be a vertex of an element $T\in \mathbb{N}_h$. Then, by taking $v_h=u_h-\delta \phi^h_p$ for some $\delta >0$ sufficiently small, and recalling that $u_h(p)>\chi_h(p)$ (see \eqref{nconth}), combined with the previous estimate $\sigma_h \leq 0 ~\text{in}~ T $, gives $\sigma_h(p)=0$ i.e.
\begin{align} \label{eq:sigp1}
\sigma_h(p)=0~ ~\forall~ p \in \cV_h ~\text{for which}~ u_h(p)>\chi_h(p),
\end{align}
{in particular,}
\begin{align*}
\sigma_h(p)=0~~ \forall\, p \in \mathcal{V}_T,~~ T \subset \mathbb{N}_h.
\end{align*}
Observe that, in view of \eqref{eq:sigp1}, the expression \eqref{sigmah0} with $w_h=\sigma_h$ is reduced to
\begin{align}\label{sigmah}
\langle\sigma_h, v_h\rangle_h&= \sum_{T \in \mathcal{T}_h} \int_T I_h(\sigma_h|_T v_h|_T)~dx=\sum_{\substack{
T \in \mathcal{T}_h\\
T\notin \mathbb{N}_h}}\sum_{p \in \mathcal{V}_T}\sigma_h(p)v_h(p)\int_T \phi_p~dx \nonumber \\
& =\sum_{T \in \mathbb{C}_h\cup\mathbb{M}_h}\frac{|T|}{d+1} \sum_{p \in \mathcal{V}_T}\sigma_h(p)v_h(p).
\end{align}

\par\noindent
We now give an approximation property for $\sigma_h$ that is proved along the lines of\cite[Lemma 3.3]{Veeser:2001:VI}.
For the sake of completeness we provide the proof.
\begin{lemma}\label{lem:Ihapprox} For any $1\leq p\leq \infty, v \in V_h$ and $T \in \mathcal{T}_h$,
\begin{equation*}
\int_T(\sigma_h v_h- I_h(\sigma_h v_h))~dx \leq h_T^2 \|\nabla \sigma_h\|_{L^{p}(T)} \|\nabla v_h\|_{L^{q}(T)},
\end{equation*}
where $q$ is the conjugate of $p$ i.e. $\frac{1}{p}+\frac{1}{q}=1$.
\end{lemma}
\begin{proof}
{
Using { a scaling argument} we find that,
\begin{align*}
\int_T(\sigma_h v_h- I_h(\sigma_h v_h))~dx & \leq h_T^2 \int_T |D^2(\sigma_h v_h)| ~dx \\
&=h_T^2 \int_T |\nabla \sigma_h \cdot \nabla v_h| ~dx,
\end{align*}
where in the last line we have used the fact that $\sigma_h$ and $v_h$ are linear functions on $T$. Finally, the estimate of the Lemma follows using the H{\"o}lder's inequality.
}
\end{proof}

%

\section{A posteriori Error Analysis: Reliability}\label{sec:Reliaibility}
In this section we lay down the scheme for the pointwise a posteriori analysis, introducing the error estimator and stating the first main result of the paper, namely, we present the reliability analysis of the error estimator. The analysis follows using with the ideas introduced in \cite{NSV:2002:Ptwise}, with adequate modifications in order to account for the non-conformity of the approximation.
Before we state the main theorem of the section and outline the path of the proof, we define the Galerkin functional, which plays the same role as the residual functional for unconstrained problems.\\

\par
\noindent
Similar to the extension of the discrete bilinear form  in  \eqref{ext:A}, we first introduce an extended definition of the continuous $\sigma(u)$ defined in \eqref{eq:sigmadef}.
Define
$\tilde{\sigma}(u) \in \cM^*$ as
\begin{align}\label{eq:extend:sigmadef}
\langle \tilde{\sigma}(u),v\rangle &=(f,v)-\widetilde{\mathcal{A}}_h(u,v),\quad \forall v \in \cM \;,
\end{align}
where the space $\cM$ is defined in \eqref{def:M}. Observe that due to the definition of $\widetilde{\mathcal{A}}_h(\cdot,\cdot)$, we plainly have
$$
\langle \tilde{\sigma}(u),v\rangle \equiv \langle \sigma(u),v\rangle \quad \forall\, v\in V.
$$

\noindent
 We now define the Galerkin or {\it residual functional} $G_h \in \cM^*$ for the nonlinear problem under consideration.
Let  $G_h \in \cM^*$ be defined as
\begin{equation}\label{eq:Gh}
\langle G_h,v \rangle =\widetilde{\mathcal{A}}_h(u-u_h,v)+\langle \tilde{\sigma}(u)-{\sigma_h},v \rangle \quad \forall v \in \cM.
\end{equation}
{
Note that since $\sigma_h \in V_h \subset L^2(\Omega)\subset H^{-1}(\Omega) $, we have $\langle \sigma_h,v \rangle =( \sigma_h,v)$ for all $v \in \cM$.

}
\par\noindent
Observe that by restricting the action of the functional $G_h$ to the subspace $H^{1}_0(\O)\subset W^{1,q}_0(\Omega)$ (and respectively to conforming subspace $V_h^{conf} \subset V_h$) we have the quasi-orthogonality property and for any $ v \in V $,
 \begin{align}\label{eq:Gh2}
 \langle G_h,v \rangle & =\widetilde{\mathcal{A}}_h(u-u_h,v)+\langle \tilde{\sigma}(u)-{\sigma_h},v \rangle  \nonumber&& \\
&=a(u,v)-\widetilde{\mathcal{A}}_h({u_h},v)+\langle \sigma(u)-{\sigma_h},v \rangle  && \nonumber\\
&=(f,v)-\widetilde{\mathcal{A}}_h(u_h,v)-\langle {\sigma_h},v \rangle .&&
\end{align}
The functional equations \eqref{eq:Gh} and  \eqref{eq:Gh2} will be essential for our analysis.

\par
\noindent
We now define the error estimators:
\begin{align}
\eta_1 &= \displaystyle \max_{T \in \cT_h}~\|h_T^2(f-\sigma_h)\|_{L^{\infty}(T)}, \label{eq:eta01}\\
\eta_2 &= \displaystyle \max_{e \in \Eho}  ~\|h_e \sjump{\nabla u_h}\|_{L^{\infty}(e)}, \label{eq:eta02}\\
\eta_3 &= \|\sjump{u_h}\|_{L^{\infty}(\cE_h)} ,\label{eq:eta03}\\
\eta_4 &={  \displaystyle \max_{T \in \mathbb{C}_h\cup \mathbb{M}_h} \|h_T^{2}\nabla \sigma_h\|_{L^{d}(T)}} ,\label{eq:eta04}\\
\eta_5 &= \|(\chi-u_h)^{+}\|_{L^{\infty}(\Omega)},\label{eq:eta05}\\
\eta_6 &=\|({u_h}-\chi)^{+}\|_{L^{\infty}(\{\sigma_h<0\})}.\label{eq:eta06}
\end{align}
The full a posteriori error estimator $\eta_h$ is then defined as:
\begin{align*}
\eta_h&=  |\log h_{min}| \Big(\eta_1 +\eta_2  +\eta_3+ \eta_4 \Big)+ \eta_5 +\eta_6.
\end{align*}
 Therein, the estimator term  $\eta_4$ can be controlled by the volume residual term up to the oscillation of $f$ as follows.
\begin{lemma} \label{lem:Estmsig}
For any $T \in \mathbb{C}_h\cup \mathbb{M}_h$, it holds that,
\begin{align*}
\| h_T^2\nabla \sigma_h\|_{L^{d}(T)} \leq \|h_T^2(f-\sigma_h)\|_{L^{\infty}(T)}+ Osc(f,T),
\end{align*}
where 
\begin{equation}\label{def:Osc}
Osc(f,T)=\displaystyle\min_{\tilde{f} \in \mathbb{P}^0(T)}h_T^2\|f- \tilde{f}\|_{L^{\infty}(T)}.
\end{equation}
\end{lemma}
\begin{proof}
Let $T \in \mathbb{C}_h\cup \mathbb{M}_h$ and $\bar{f} \in \mathbb{P}^0(T)$. A use of inverse inequality \eqref{eq:inverse1} and triangle inequality yields
\begin{align*}
\| h_T^2\nabla \sigma_h\|_{L^{d}(T)} &= \| h_T^2\nabla (\sigma_h-\bar{f})\|_{L^{d}(T)} \\
&\leq  \|h_T^2(\bar{f}-\sigma_h)\|_{L^{\infty}(T)} \\& \leq \|h_T^2(f-\sigma_h)\|_{L^{\infty}(T)}+ Osc(f,T).
\end{align*}
\end{proof}
\par\noindent
We now can state the main result of this section, namely the reliability of the error estimator $\eta_h$.
\begin{theorem}\label{thm:reliability}
Let $u \in \cK$ and $u_h \in \cK_h$ be the solutions of \eqref{eq:MP} and \eqref{eq:Ah}, respectively.
Then,
\begin{align*}
\|u-{u_h}\|_{L^{\infty}(\Omega)}\lesssim \eta_h.
\end{align*}
\end{theorem}
\noindent
The proof of this theorem  is done in several steps and will be carried out in the following subsections. As mentioned before, the analysis is motivated by the ideas laid in \cite{NSV:2002:Ptwise}, modifying accordingly to account for the non-conformity of the approximation.
To avoid having undesirable restrictions on the mesh, such as weakly acute, we refrain ourselves from using any discrete maximum principle, and we use instead the stability of the continuous problem and the continuous maximum principle. To do so, we construct sub and super solutions for the continuous solution, starting from the conforming part of the approximate solution $u_h$, i.e, $E_hu_h$ and correcting it appropriately. Throughout the rest of the paper,  $u_h^{conf}$  always denotes $E_hu_h$.
The rationality of basing the construction on  the conforming part $u_h^{conf}$  stem from the fact that we aim at applying the stability of the continuous problem, and so we need the barriers to belong to $H^{1}(\O)$. Once the sub and super solutions are constructed the continuous maximum principle provides a first bound on the pointwise error. This is done in subsection \ref{barrier}.

The construction of the lower and upper barriers involves a corrector function, denoted by $w$,  that accounts for the consistency  and the nonconformity errors.  The analysis of the reliability of the estimator, namely  the proof of Theorem \ref{thm:reliability}, is then completed by providing
a maximum norm estimate for the corrector function in terms of the local estimators $\eta_i, i=1,2,3,4$ given in \eqref{eq:eta01}, \eqref{eq:eta02}, \eqref{eq:eta03} and \eqref{eq:eta04}. Such bound is given in subsection \ref{subsec:boundw}, and it is proved using fine regularity estimates of the Green's function associated to the unconstrained Poisson problem. Here, our approach differs from that presented in \cite{NSV:2002:Ptwise}, and simplifies somehow the analysis.

\par
\subsection{Sub and Super Solutions}\label{barrier}
We introduce the following barrier functions of the exact solution. We define $u^{*} $ and $u_{*}$ as,
\begin{align}
u^{*}&=u_h^{conf}+w+\|w\|_{L^{\infty}(\Omega)}+\|(\chi-u_h^{conf})^{+}\|_{L^{\infty}(\Omega)},  &&  \label{eq:upper} \\
u_{*}&= u_h^{conf}+w-\|w\|_{L^{\infty}(\Omega)}-\|(u_h^{conf}-\chi)^{+}\|_{L^{\infty}(\{\sigma_h<0\})},  && \label{eq:lower}
\end{align}
where $ w \in H^1_0(\Omega)$ is the weak solution of the Poisson problem with right hand side functional defined as
\begin{align}\label{eq:Fhdef00}
\cF_h(v)&=\langle G_h,v\rangle +\widetilde{\cA}_h(u_h-u_h^{conf},v) \qquad\forall v \in  W^{1,q}_0(\Omega), ~1\leq q <(d/d-1),
\end{align}

i.e. $ w \in H^1_0(\Omega)$ solves the following linear problem
\begin{align}\label{eq:GF}
\int_{\Omega}\nabla w\cdot \nabla v~dx= \cF_h(v)  \quad \forall v \in H^1_0(\Omega).
\end{align}

\noindent

The existence of $w$ is guaranteed by the Riesz representation theorem.
Notice that $\cF_h$ is well defined as a functional in the dual of $ W^{1,q}_0(\Omega)$.
We also observe that when acting on more regular functions, say $v \in H^1_0(\Omega), \,\, \cF_h$ reduces to
\begin{align}\label{eq:Fhdef}
\cF_h(v)&=\langle G_h,v\rangle +\widetilde{\cA}_h(u_h-u_h^{conf},v)\\
&=\widetilde{\mathcal{A}}_h(u-u_h^{conf},v)+\langle {\sigma}(u)-{\sigma_h},v \rangle \notag\\
&=\int_{\Omega}\nabla(u-u_h^{conf})\cdot\nabla v~dx+\langle {\sigma}(u)-{\sigma_h},v \rangle \notag.
\end{align}
Next, we prove that $u_{*}$ and $u^{*}$ are sub and super solutions of the solution of variational inequality \eqref{eq:MP}. This proof will also establish the essence of having the conforming part of $u_h$ in the definitions of the barrier functions  so as to use the stability of the continuous problem, which in turn forces us to use functions in $H^1_0(\Omega)$. 
{
We would like to remark  that in \cite{NSV:2002:Ptwise}, the definition of Galerkin functional $G_h$ entailed a more refined  $\tilde\sigma_h$ which used the positivity preserving operator from \cite{CN:2000:VI} and allow for suitable cancellation, providing hence the localization of the error estimator. Here, thanks to the local nature of discontinuous space and the particular construction of the averaging operator $E_h$ in \eqref{eq:Enrich}, we can complete the proof of Lemma \ref{lem:barrier} below by sticking to the definition \eqref{def-sigmah} of $\sigma_h$  in the expression of Galerkin functional $G_h$.
}

\begin{lemma} \label{lem:barrier}
Let u be the solution of \eqref{eq:MP} and $u^{*}, u_{*}$ be the upper and lower barriers defined in equations \eqref{eq:upper} and \eqref{eq:lower}, respectively. Then, it holds that
\begin{align*}
  u &\leq u^{*}, \\
u_{*}&\leq u.
\end{align*}
\end{lemma}
{
\begin{remark} 

The sub solution in \eqref{eq:lower}  could be alternatively defined by
$$
u_{*}= u_h^{conf}+w-\|w\|_{L^{\infty}(\Omega)}-\|(u_h-\chi)^{+}\|_{L^{\infty}(\Lambda_h^1)} -\|(u_h^{conf}-\chi)^{+}\|_{L^{\infty}(\Lambda_h^2)}
$$
where $\Lambda_h^1= \{T \in \cT_h: T \cap \partial \Omega=\emptyset ~\text{and}~  \sigma_h<0 ~\text{on}~ T\}$ and $\Lambda_h^2= \{T \in \cT_h: T \cap \partial \Omega\neq \emptyset ~\text{and}~  \sigma_h<0 ~\text{on}~ T\}$ and Lemma \ref{lem:barrier} would still hold true.
\end{remark}
}

\begin{proof}
First we show that $u \leq u^*$.
Set $ v=(u-u^{*})^{+}$ and we claim to prove that $v=0$. {Note that, $u^{*}$ is defined using the conforming part of $u_h$ i.e. $u_h^{conf}$ which guarantees that $v \in H^1(\Omega)$}. We have
\begin{eqnarray*}
(u-u^{*})|_{\partial\Omega}\leq (u-u_h^{conf})|_{\partial\Omega}-\|(\chi-u_h^{conf})^{+}\|_{L^{\infty}(\Omega)} \leq 0,
\end{eqnarray*}
which implies~ $
v|_{\partial\Omega}=0$, therefore $v \in H_0^1(\Omega)$. It then suffices to show that~ $ \|\nabla v\|_{L^2(\Omega)}= 0$. A use of \eqref{eq:GF} together with definitions of $\cF_h$ and $G_h$ yields,
\begin{align*}
\|\nabla v\|^2_{L^2(\Omega)}&=\int_{\Omega}\nabla(u-u_h^{conf})\cdot \nabla v~dx -\int_{\Omega}\nabla w \cdot \nabla v~dx\\
&=\int_{\Omega}\nabla(u-u_h^{conf})\cdot \nabla v~dx -\cF_h(v)\\
&=\int_{\Omega}\nabla(u-u_h^{conf})\cdot \nabla v~dx -\langle G_h,v\rangle -\widetilde{\cA}_h(u_h-u_h^{conf},v) \\
&=\widetilde{\cA}_h(u-{u}_h,v)-\langle G_h,v\rangle \\
&= -\langle\sigma(u)-\sigma_h,v \rangle \leq -\langle{\sigma(u)},v\rangle,
\end{align*}
 where in the last step we have used the property $\sigma_h\leq 0$ in $T~ \forall~ T \in \cT_h$ . We arrive at desired claim if we prove that $ \langle{\sigma(u)},v \rangle=0$.
Since ~ $supp(\sigma(u))\subset \{u=\chi\} $ and
$\{v>0\}\subset \{u>\chi\}$ ~ because
\begin{align*}
v>0 \implies
u>u^{*}&\geq u_h^{conf} + \|(\chi-u_h^{conf})^{+}\|_{L^{\infty}(\Omega)} \\
&=\chi-(\chi-u_h^{conf})+ \|(\chi-u_h^{conf})^{+}\|_{L^{\infty}(\Omega)} \geq \chi.
\end{align*}
Hence we obtain $\langle\sigma(u),v\rangle=\int_{\Omega} v \,d\sigma =0$.
Thus $\|\nabla v\|_{L^2(\Omega)}= 0$ and a use of Poincare inequality concludes  $u \leq u^*$.\\

\par \noindent
 Next, we show that $ u_{*}\leq u $. Let $v=(u_{*}-u)^{+} $, then it suffices to prove that $v=0$.  Since $u, u_h^{conf} \in H^1_0(\Omega)$, we have
\begin{align*}
(u_{*}-u)|_{\partial\Omega}&\leq (u_h^{conf}-u)|_{\partial\Omega}-\|(u_h^{conf}-\chi)^{+}\|_{L^{\infty}(\{\sigma_h<0\})}
\leq 0.
\end{align*}
{Therefore ~~
$v|_{\partial\Omega}=0$, thus $v \in H_0^1(\Omega)$. We obtain the desired result from Poincare inequality if we show that $\|\nabla v\|_{L^2(\Omega)}= 0.$} Using definitions of $u_{*}$, $w$, $\cF_h$ and $G_h$, we find
\begin{align}\label{eq:v1}
\|\nabla v\|^2_{L^2(\Omega)}&=\int_{\Omega}\nabla(u_h^{conf}-u)\cdot\nabla v ~dx+\int_{\Omega}\nabla w \cdot\nabla v~dx \notag\\
&=\int_{\Omega}\nabla(u_h^{conf}-u)\cdot\nabla v~dx +\cF_h(v) \notag\\
&=\int_{\Omega}\nabla(u_h^{conf}-u)\cdot\nabla v~dx +\langle G_h,v\rangle +\widetilde{\cA}_h(u_h-u_h^{conf},v) \notag\\
&=\widetilde{\cA}_h({u_h}-u,v) +\langle G_h,v \rangle \notag\\
&= \langle \sigma(u)-\sigma_h,v \rangle
\leq -(\sigma_h,v)=-\sum_{T \in \mathcal{\cT}_h}\int_T \sigma_h v~dx \qquad\quad  (\text{ since }\sigma(u)\leq 0).
\end{align}
\noindent
Now to obtain the desired claim, it suffices to show that $(\sigma_h,v)=0$, which we prove in the following two steps.\\

\textbf{STEP 1}: Therein, we show that if there exists $ x \in int(T)~ s.t.~ v(x)>0 ~ i.e.~ u_{*}(x)>u(x)~ \text{then} ~\sigma_h|_T=0$. We prove it by contradiction.
Suppose $\exists ~x \in ~int(T)~ s.t. ~u_{*}(x)>u(x) ~\text{and} ~\sigma_h(x)<0$.~Then, using \eqref{eq:lower} and $u\geq \chi$, we find
\begin{align*}
u_h^{conf}(x)&>u(x)+(\|w\|_{L^{\infty}(\Omega)}-w(x)) +\|(u_h^{conf}-\chi)^{+}\|_{L^{\infty}(\{\sigma_h<0\})} \\
&> \chi(x)+\|(u_h^{conf}-\chi)^{+}\|_{L^{\infty}(\{\sigma_h<0\})} \geq u_h^{conf}(x),
\end{align*}
which is a contradiction.
Thus $\sigma_h(x)=0$. Since $\sigma_h \leq 0$ in $T$ and it is linear in $T$, which implies $\sigma_h|_T=0$.\\
\\

\textbf{STEP 2}: We have,
\begin{eqnarray*}
(\sigma_h,v)=\sum_{T \in \mathcal{\cT}_h}\int_T \sigma_h v~dx.
\end{eqnarray*}
Let $ T\subset \cT_h$ be arbitrary. If $v=0$ in $T$ then $\int_T \sigma_h v ~dx=0$.
 If not, then $\exists~ x \in int(T)~~\text{such that}~~ v(x)>0$. A use of {\textbf{STEP 1}} yields $\sigma_h|_T=0 ~~\text{and thus} ~~\int_T \sigma_h v~dx =0$.\\
 \par \noindent
 Therefore, from equation \eqref{eq:v1}, we get $\|\nabla v\|_{L^2(\Omega)}= 0$ and a use of Poincare inequality ensures $u_{*} \leq u$.\\
\par \noindent
This completes the proof of this lemma.

\end{proof}

\noindent
\par
\subsection{Pointwise Estimate on the corrector function $\|w\|_{L^{\infty}(\Omega)}$:} \label{subsec:boundw}
 We now provide  an estimate on the supremum norm of $w$  in terms of the local error indicators \eqref{eq:eta01}-\eqref{eq:eta02}-\eqref{eq:eta03}-\eqref{eq:eta04}.
As is typical in the maximum-norm analysis of finite elements, we employ the Green's function of the unconstrained Poisson problem. Rather than using the approach in \cite{NSV:2002:Ptwise} employing  a regularized Green's function, we follow the path from \cite{DG:2012:Ptwise2, GG} by considering the Green's function singular with respect to the point $x_0\in \Omega$ for which $w$ attains its maximum. This allows for a better account of the nonconformity of the method (see Remark \ref{rem:NONSYM}) and a slight improvement in the estimate (by one power less in the logarithmic factor), since we benefit from the fine estimates collected in Proposition \ref{Greens0}  that account for the different regularity (of the Green's function) near and far from the singularity.
\begin{proposition}\label{lem:Linfw}
{Let $w \in H^1_0(\Omega)$ be the function defined in \eqref{eq:GF}. There holds,}
\begin{align*}
\|w\|_{L^{\infty}(\Omega)}& \lesssim |\log h_{min}|( \eta_1 +\eta_2 +\eta_3 +\eta_4 ).
\end{align*}
\end{proposition}
\begin{proof}
 Let $x_0\in \Omega\smallsetminus \partial \Omega$ be the point at which the maximum of the function $w$ is attained, i.e., $|w(x_0)|=\|w\|_{L^{\infty}(\O)}$. Then, there is an element $T_0\in \Th$ such that its closure contains $x_0\in \overline{T_0}$.
 Let $ \mathcal{G}(x_{0},\cdot)$ be the Green's function with singularity at $x_0$, as defined in \eqref{Ga0} for $x=x_0$.
 First, notice that from the definition of $w$ in \eqref{eq:GF}-\eqref{eq:Fhdef} and having defined the functional ${\calF_h\in (W^{1,q}_0(\O))^{*}}$ $1\leq q < \frac{d}{d-1}$, it follows that $w$ is indeed  more regular,  and so in view of \eqref{key0} and and the definition in \eqref{eq:GF}-\eqref{eq:Fhdef} we have
\begin{equation}\label{key02}
\begin{aligned}
w(x_0) &=\int_{\O}w(\xi)\delta_{x_0}(\xi) d\xi =(\nabla w, \nabla \mathcal{G}(x_0,\cdot))=\calF_h( \mathcal{G}(x_0,\cdot)) &&\\
&=\langle G_h,\mathcal{G} \rangle+\widetilde{\cA}_h(u_h-u^{conf}_h, \mathcal{G})\;, &&
\end{aligned}
\end{equation}
where we have dropped the dependence on $x_0$ (which is now fixed) in the last terms above to simplify the notation.
Therefore, to provide the estimate on $|w(x_0)|=\|w\|_{L^{\infty}(\O)}$, we need to bound the terms of the right hand side of the last equation.   Let $\omega_{0}=\omega_{T_0}$ be the set of elements touching $T_0$ (as defined in \eqref{omega0}) and let $\omega_1$ be the patch of elements touching $\omega_0$ i.e., the set of elements with non empty intersection with $T' \in \omega_0$ (we denote $\omega_1$ to avoid the cumbersome notation $\omega_{\omega_{T_0}}$).
In our estimates, we will consider a disjoint decomposition of the finite element partition $\Th=(\Th \cap \omega_0 )\cup (\Th \smallsetminus \omega_0)$, which will allow us  to use the localized regularity estimates for the Green's function stated in Proposition \ref{Greens0}.\\
\par\noindent
We now estimate the terms in the right hand side of \eqref{key02}.
For the first one, we set $\mathcal{G}_h=\Pi_h(\mathcal{G}) \in V_h$ and write
\begin{align}\label{eq:EG1}
\langle G_h,\mathcal{G} \rangle &= \langle G_h,\mathcal{G}-\Pi_h(\mathcal{G}) \rangle+\langle G_h,\Pi_h(\mathcal{G})\rangle.
\end{align}
Now, we handle the two terms of the right hand side of above equation as follows:
from the definition of $G_h$, $\tilde{\sigma}(u)$, $\sigma(u)$ and $\sigma_h$, we find,
\begin{align} \label{eq:EG2}
\langle G_h,\Pi_h(\mathcal{G}) \rangle&=\widetilde{\cA}_h(u-{u_h},\Pi_h(\mathcal{G}))+\langle \tilde{\sigma}(u)-{\sigma_h},\Pi_h(\mathcal{G}) \rangle \notag\\
&=(f,\Pi_h(\mathcal{G}))-\widetilde{\cA}_h(u_h,\Pi_h(\mathcal{G}))-(\sigma_h,\Pi_h(\mathcal{G}))\notag\\
&=(f,\Pi_h(\mathcal{G}))-\cA_h(u_h,\Pi_h(\mathcal{G}))-(\sigma_h,\Pi_h(\mathcal{G}))\notag\\
&=\langle \sigma_h,\Pi_h(\mathcal{G}) \rangle_h-(\sigma_h,\Pi_h(\mathcal{G})).
\end{align}
 where we have used that $\widetilde{\cA}_h(v,w)=\cA_h(v,w), \forall v,w \in V_h$. Equation \eqref{eq:Gh} together with the definition of $\tilde{\sigma}(u)$ gives,
\begin{align} \label{eq:EG3}
\langle G_h,\mathcal{G}-\Pi_h(\mathcal{G}) \rangle&=\widetilde{\cA}_h(u-u_h,\mathcal{G}-\Pi_h(\mathcal{G}))+\langle \tilde{\sigma}(u)-{\sigma_h},\mathcal{G}-\Pi_h(\mathcal{G}) \rangle\notag\\
&=(f,\mathcal{G}-\Pi_h(\mathcal{G}))-\widetilde{\cA}_h(u_h,\mathcal{G}-\Pi_h(\mathcal{G}))-(\sigma_h,\mathcal{G}-\Pi_h(\mathcal{G}))\;. 
\end{align}
Therefore, plugging \eqref{eq:EG2} and \eqref{eq:EG3}  into \eqref{eq:EG1} and substituting the result into \eqref{key02}  we have
\begin{align}\label{key4}
\calF_h( \mathcal{G})  &=(f-\sigma_h,\mathcal{G}-\Pi_h(\mathcal{G})) + \langle \sigma_h,\Pi_h(\mathcal{G}) \rangle_h-(\sigma_h,\Pi_h(\mathcal{G}))
 -\widetilde{\cA}_h(u_h,\mathcal{G}-\Pi_h(\mathcal{G}))  &&\nonumber\\& \quad+\widetilde{\cA}_h(u_h-u_h^{conf},\mathcal{G}) &&
\end{align}
For the first  term above, H\"older inequality gives
\begin{align*}
\left|(f-\sigma_h,\mathcal{G}-\Pi_h(\mathcal{G}))\right| &\lesssim \sum_{T \in \cT_h} \| h_{T}^{2}( f-\sigma_h)\|_{L^{\infty}(T)} \| h_{T}^{-2}(\mathcal{G}-\Pi_h \mathcal{G})\|_{L^{1}(T)} \\
&\lesssim\max_{T\in \Th}  \| h_{T}^{2}( f-\sigma_h)\|_{L^{\infty}(T)} \sum_{T \in \cT_h}  \| h_{T}^{-2}(\mathcal{G}-\Pi_h \mathcal{G})\|_{L^{1}(T)}\;.
\end{align*}
Using now the approximation estimate \eqref{stabPik2} from Lemma \ref{lem:approx_proj},  we obtain
\begin{align}\label{coso2}
 \sum_{T \in \cT_h}  \| h_{T}^{-2}(\mathcal{G}-\Pi_h \mathcal{G})\|_{L^{1}(T)} &= \sum_{T \in \Th \cap \omega_0}  \| h_{T}^{-2}(\mathcal{G}-\Pi_h \mathcal{G})\|_{L^{1}(T)} + \sum_{T \in \Th \smallsetminus \omega_0}  \| h_{T}^{-2}(\mathcal{G}-\Pi_h \mathcal{G})\|_{L^{1}(T)} \nonumber &&\\
&\lesssim \sum_{T \in \Th \cap \omega_0}   h_{T}^{-1} |\mathcal{G}|_{W^{1,1}(T)} +\sum_{T \in \Th \smallsetminus \omega_0}  |\mathcal{G}|_{W^{2,1}(T)}\;. &&
\end{align}
Hence, recalling the definition of $\eta_1$ from \eqref{eq:eta01}, we have
\begin{equation}\label{primo}
\left|(f-\sigma_h,\mathcal{G}-\Pi_h(\mathcal{G}))\right| \lesssim~\eta_1~\left(\sum_{T \in \Th \cap \omega_0}   h_{T}^{-1} |\mathcal{G}|_{W^{1,1}(T)} +\sum_{T \in \Th \smallsetminus \omega_0}  |\mathcal{G}|_{W^{2,1}(T)}\right)\;.
\end{equation}
For the second and third terms in \eqref{key4}, recalling the localization property observed in \eqref{sigmah}, we first set $\Th^{C}=\mathbb{C}_h \cup \mathbb{M}_h$ and use  H\"older inequality  but splitting the contributions in the disjoint decomposition of the  partition  $( \Th^{C} \smallsetminus \omega_0) \cup (\Th^{C}\cap \omega_0)$. In particular, using also Lemma \ref{lem:Ihapprox}, Sobolev imbedding (or even the inverse  inequality  \eqref{eq:inverse}) together with the stability estimate \eqref{stabPik0} from Lemma \ref{lem:approx_proj} we obtain,
\begin{align}
&\left| (\sigma_h,\Pi_h \mathcal{G})-\langle \sigma_h,\Pi_h \mathcal{G} \rangle_h \right|   \lesssim   \left( \sum_{T \in \Th^C \smallsetminus \omega_0 } \|  h_T^2 \nabla \sigma_h\|_{L^{d}(T)} \|\nabla \Pi_h\mathcal{G}\|_{L^{d/(d-1)}(T)}  \right. &&\notag \\
&\qquad\qquad\qquad \qquad \qquad \qquad\qquad \left. +\sum_{T \subset  \Th^C \cap \omega_0 }  \| h_T^2\nabla \sigma_h\|_{L^{d}(T)}  h_{T}^{-1} \|\nabla \Pi_h\mathcal{G}\|_{L^{1}(T)} \right)&&\notag\\
&\qquad \quad   \lesssim   \left( \sum_{T \in \Th^C \smallsetminus \omega_0 } \|  h_T^2 \nabla \sigma_h\|_{L^{d}(T)} \|\nabla \Pi_h\mathcal{G}\|_{W^{1,1}(T)} +\sum_{T \subset  \Th^C \cap \omega_0 }  \| h_T^2\nabla \sigma_h\|_{L^{d}(T)}  h_{T}^{-1} \|\nabla \mathcal{G}\|_{L^{1}(T)} \right) && \notag\\
&\qquad \quad   \lesssim  \max_{T \in \Th^C } \|  h_T^2 \nabla \sigma_h\|_{L^{d}(T)} \left( \sum_{T \in \Th^C \smallsetminus \omega_0 }\big( |\Pi_h \mathcal{G}|_{W^{1,1}(T)}+|\Pi_h \mathcal{G}|_{W^{2,1}(T)}  \big)+\sum_{T \subset  \Th^C \cap \omega_0 } h_{T}^{-1} \|\nabla \mathcal{G}\|_{L^{1}(T)} \right) \notag &&	\\
   &\qquad \qquad \lesssim  ~\eta_4~\left( { |\mathcal{G}|_{W^{1,1}(\Th^C \smallsetminus \omega_0)}}+|\mathcal{G}|_{W^{2,1}(\Th^C \smallsetminus \omega_0)}  +\sum_{T \subset  \Th^C \cap \omega_0 } h_{T}^{-1} \|\nabla \mathcal{G}\|_{L^{1}(T)} \right), &&\label{sigmaEst}
  \end{align}
where in the last step we have used the definition of $\eta_4$ from \eqref{eq:eta04}.\\
For the fourth term in \eqref{key4}, using the fact that $\Pi_h$ is a projection (hence $\Pi_h(\mathcal{G}-\Pi_h(\mathcal{G}))=0$) and integrating by parts taking into account that $u_h \in \mathbb{P}^{1}(\Th)$ we find,
\begin{align*}
&\widetilde{\cA}_h(u_h,\mathcal{G}-\Pi_h(\mathcal{G})) =a_h(u_h,(\mathcal{G}-\Pi_h(\mathcal{G})))+\tilde{b}_h(u_h,(\mathcal{G}-\Pi_h(\mathcal{G}))) &&\\
&\quad =\sum_{T\in\cT_h} \int_{\partial T}\frac{\partial u_h|_T}{\partial n_T}(\mathcal{G}-\Pi_h(\mathcal{G}))\,ds+\tilde{b}_h(u_h,(\mathcal{G}-\Pi_h(\mathcal{G}))) &&\\
&\quad =\sum_{e\in \Eho} \int_e\sjump{\nabla u_h}\smean{\mathcal{G}-\Pi_h(\mathcal{G})}\,ds +\sum_{e\in\cE_h} \int_e
\frac{\gamma}{h_e}\sjump{u_h}\sjump{\mathcal{G}-\Pi_h(\mathcal{G})}\,ds. &&
\end{align*}
H\"older inequality and trace inequality \eqref{eq:traceAg} together with the localized estimates on $\mathcal{G}$ gives
\begin{align*}
& |\widetilde{\cA}_h(u_h,\mathcal{G}-\Pi_h(\mathcal{G})) | \lesssim \sum_{e \in \Eho}  \| h_e \sjump{\nabla u_h} \|_{L^{\infty}(e)} h_e^{-1} \|\smean{\mathcal{G}-\Pi_h \mathcal{G}}\|_{L^{1}(e)} \\
& \quad+ \|\sjump{u_h} \|_{L^{\infty}(\Eh)}\sum_{e \in \Eh} h_e^{-1}\|\sjump{\mathcal{G}-\Pi_h \mathcal{G}}\|_{L^{1}(e)} &&\\
&\quad   \lesssim  \left( \max_{e \in \Eho}  \| h_e \sjump{\nabla u_h} \|_{L^{\infty}(e)} + \|\sjump{u_h} \|_{L^{\infty}(\Eh)}\right) \sum_{T\in \Th}  \left(  h_T^{-2}\|\mathcal{G}-\Pi_h \mathcal{G}\|_{L^{1}(T)} +h_T^{-1}{\|\nabla( \mathcal{G}-\Pi_h \mathcal{G})\|_{L^{1}(T)} }\right). &&
\end{align*}
Arguing as in \eqref{coso2}, and using the definitions from \eqref{eq:eta02} and \eqref{eq:eta03} we get,
\begin{equation}\label{terzo}
 |\widetilde{\cA}_h(u_h,\mathcal{G}-\Pi_h(\mathcal{G})) | \lesssim  ({\eta_2 +\eta_3} ) ~\left(\sum_{T \in \Th \cap \omega_0}   h_{T}^{-1} |\mathcal{G}|_{W^{1,1}(T)} +\sum_{T \in \Th \smallsetminus \omega_0}  |\mathcal{G}|_{W^{2,1}(T)}\right)\;.
\end{equation}
For the very last term in \eqref{key4}, following \cite{DG:2012:Ptwise2}, we also introduce a continuous but piecewise linear function $\xi \in V_h^{conf}$ that is identically one on $\omega_0$ and is identically zero at the nodes in $\Th\smallsetminus \overline{\omega_0}$, so that it's support is contained in $\overline{\omega_1}$, i.e,  $\supp(\xi) \subseteq \overline{\omega_1}$ while the  complementary function $\supp(1-\xi) \subset \Omega \smallsetminus \overline{\omega_0}$. This function will be used as a cut-off function and  allow us to localize the terms and  to integrate by parts in one of them. Notice that    $\| \xi\|_{L^{\infty}(\Th)}=1$ and  $\|\nabla \xi\|_{L^{\infty}(\Th)} \lesssim \delta^{-1}$ with $ \delta$ being the maximum diameter of the support of $\xi$.

\par \noindent
We write, $u_h-u_h^{conf}=\xi(u_h-u_h^{conf})+ (1-\xi)(u_h-u_h^{conf})=\xi^0+\xi^1$ and then,
\begin{align}\label{eq:AhEst}
\widetilde{\cA}_h(u_h-u_h^{conf},\mathcal{G}) &=\widetilde{\cA}_h(\xi( u_h-u_h^{conf}),\mathcal{G})+ \widetilde{\cA}_h( (1-\xi) (u_h-u_h^{conf}),\mathcal{G}) \notag&&\\
&=\widetilde{\cA}_h(\xi^0,\mathcal{G})+ \widetilde{\cA}_h(\xi^{1},\mathcal{G}). &&
\end{align}
\par \noindent
First, we consider the second term of the last equation.
Integrating by parts and using the facts that on the support of $(1-\xi)$ the Green's function $\mathcal{G}$ is harmonic and  satisfies $\jump{\mathcal{G}}=0$ and $\jump{\nabla \mathcal{G}}=0$ on any $e\in \Eh\cap (\Omega \smallsetminus \omega_0)$ we find,
\begin{align}
 \widetilde{\cA}_h(\xi^{1},\mathcal{G}) &= (\nabla_h \xi^1, \nabla \mathcal{G}) - \sum_{e\in\cE_h } \int_e \smean{\nabla \Pi_h(\mathcal{G})}\sjump{\xi^1} ds && \label{eq:NONSYM}\\
 &=( \xi^1, -\Delta \mathcal{G}) +\sum_{e\in \Eh} \int_e\sjump{\xi^1}\smean{\nabla \mathcal{G}}\,ds - \sum_{e\in\cE_h} \int_e \smean{\nabla \Pi_h(\mathcal{G})}\sjump{\xi^1} ds &&\notag\\
 & =\sum_{e\in \Eh \cap (\Omega \smallsetminus \omega_0)} \int_e\sjump{\xi^1}\smean{\nabla (\mathcal{G}-  \Pi_h(\mathcal{G}))}\,ds \notag\\
 &=\sum_{e\in \Eh \cap (\Omega \smallsetminus \omega_0)} \int_e\sjump{(1-\xi) u_h}\smean{\nabla (\mathcal{G}-  \Pi_h(\mathcal{G}))}\,ds, \notag
 \end{align}
hence, trace inequality \eqref{eq:traceAg} gives
\begin{align}
\left| \widetilde{\cA}_h(\xi^{1},\mathcal{G}\right| &\lesssim   \|\sjump{u_h} \|_{L^{\infty}(\Eh)}\sum_{T\in  (\Th \smallsetminus \omega_0)} \left( h^{-1}_{T}\|  \nabla( \mathcal{G}-\Pi_h \mathcal{G})\|_{L^{1}(T)} + |   \mathcal{G}-\Pi_h \mathcal{G} |_{W^{2,1}(T)} \right)&&\notag\\
  &\lesssim   \eta_3  |   \mathcal{G} |_{W^{2,1}(\Th\smallsetminus \omega_0)}.&&\label{eq:AhEst1}
 \end{align}

\par \noindent
 Finally, for the term near the singularity, taking into account the properties of $\xi$ and using inverse estimates, H\"older inequality, trace inequality \eqref{eq:traceAg}, the approximation results \eqref{eq:apen3} and Lemma \ref{lem:Eapen4} together with inverse inequality \eqref{eq:inverse1} and the stability estimates \eqref{stabPik0} from Lemma \ref{lem:approx_proj} we find,
 \begin{align}\label{eq:AhEst2}
 \left| \widetilde{\cA}_h(\xi^{0},\mathcal{G}) \right|&=\left| \int_{\omega_1 }\nabla_h \xi^0  \cdot \nabla \mathcal{G}dx -\sum_{e\in\Eh} \int_e \smean{\nabla \Pi_h(\mathcal{G})}\sjump{\xi^0} ds\right| && \notag\\
   &=\left| \int_{\omega_1 }\nabla_h (\xi (u_h-u_h^{conf})) \cdot \nabla \mathcal{G} dx - \sum_{e\in\cE_h} \int_e \smean{\nabla \Pi_h(\mathcal{G})}\sjump{\xi (u_h-u_h^{conf})} ds \right|&&\notag \\
  &\lesssim \sum_{T\in \omega_1} \left( \delta^{-1} \|u_h-u^{conf}_h\|_{L^{\infty}(T)}+ \|\nabla (u_h-u^{conf}_h)\|_{L^{\infty}(T)}\right) \|\nabla \mathcal{G}\|_{L^{1}(T)} \notag \\
  & \quad +\sum_{e\in\Eh\cap \omega_1} \|\jump{u_h}  \|_{L^{\infty}(e)} \|\smean{\nabla \Pi_h  \mathcal{G}}\|_{L^{1}(e)}&& \notag \\
  &\lesssim   \|\jump{u_h}\|_{L^{\infty}(\Eh\cap \omega_1)}\left(\sum_{T\in \omega_1}  (\delta^{-1} +h_{T}^{-1} ) \|\nabla \mathcal{G}\|_{L^{1}(T)} +\sum_{e\in\Eh\cap \omega_1}\|\smean{\nabla \Pi_h  \mathcal{G}}\|_{L^{1}(e)} \right) && \notag \\
  &\lesssim \eta_3  \left( \sum_{T\in \omega_1} (\delta^{-1} +h_{T}^{-1} )   \|\nabla \mathcal{G}\|_{L^{1}(T)} +h_{T}^{-1}\|\nabla \Pi_h  \mathcal{G}\|_{L^{1}(T)} + |\Pi_h  \mathcal{G}|_{W^{2,1}(T)} \right)&& \notag \\
& \lesssim \eta_3  \left( \sum_{T\in \omega_1} (\delta^{-1} +h_{T}^{-1} )   \|\nabla \mathcal{G}\|_{L^{1}(T)}  \right).&
\end{align}

\par \noindent
Plugging \eqref{eq:AhEst1} and \eqref{eq:AhEst2} into\eqref{eq:AhEst}, we get
 \begin{align}\label{eq:AhEst3}
 \left| \widetilde{\cA}_h( u_h-u_h^{conf},\mathcal{G}) \right| & \lesssim \eta_3 \left( |   \mathcal{G} |_{W^{2,1}(\Th\smallsetminus \omega_0)} \right)
 +\eta_3 \left( \sum_{T\in \omega_1} (\delta^{-1} +h_{T}^{-1} )   \|\nabla \mathcal{G}\|_{L^{1}(T)} \right).
\end{align}
\par \noindent
To conclude we now note that from the shape regularity assumption we can guarantee that there exists $C_0, C_1>0$ with $C_1>C_0$ such that the balls centered at $x_0$ and with radius $C_0h_0$ and $C_1 h_0$ satisfy
  \begin{align*}
 \mathcal{B}_0 &:= \mathcal{B}(x_0, C_0 h_0) \subset \omega_0  \quad \Longrightarrow \quad \Omega \smallsetminus \omega_0 \subset \Omega \smallsetminus \mathcal{B}_0\;, &&\\
 \mathcal{B}_1 &:= \mathcal{B}(x_0, C_1 h_0) \supset \omega_1\;, &&
 \end{align*}
 and therefore the regularity estimates from Proposition \ref{Greens0} can be applied by taking the radius in the statement of the Proposition \ref{Greens0} $2\rho=C_1h_0=2C_0h_0$. Then, we have the following estimates
 \begin{align*}
 &  \sum_{T\in \omega_1} (\delta^{-1} +h_{T}^{-1} )   \|\nabla \mathcal{G}\|_{L^{1}(T)}  \lesssim h_{0}^{-1}\|\nabla \mathcal{G}\|_{L^{1}(\Omega \cap \mathcal{B}_1)}   \lesssim (1+|\log{(h_0)}|), &&
 \end{align*}
 \begin{align*}
   | \nabla  \mathcal{G} |_{W^{1,1}(\Th\smallsetminus \omega_0)} + |   \mathcal{G} |_{W^{2,1}(\Th\smallsetminus \omega_0)}  &\lesssim   |   \mathcal{G} |_{W^{2,1}(\Omega\smallsetminus \mathcal{B}_0)} \lesssim (1+|\log{(h_0)}|), \\
  \sum_{T \in \Th \cap \omega_0}   h_{T}^{-1} |\mathcal{G}|_{W^{1,1}(T)} &  \lesssim h_{0}^{-1}\|\nabla \mathcal{G}\|_{L^{1}(\Omega \cap \mathcal{B}_1)} \lesssim 1,&&\\
   |\mathcal{G}|_{W^{2,1}(\Th \smallsetminus \omega_0)} & \lesssim  |   \mathcal{G} |_{W^{2,1}(\Omega\smallsetminus \mathcal{B}_0)}  \lesssim (1+\log(h_0))&&\\
  { |\mathcal{G}|_{W^{1,1}(\Th^C \smallsetminus \omega_0)} }&\lesssim 1.
\end{align*}
\par \noindent
 Combining these estimates on the Green's function together with equations \eqref{key02}, \eqref{key4},
 \eqref{primo}, \eqref{sigmaEst}, \eqref{terzo} and \eqref{eq:AhEst3}, we get the desired result.

\end{proof}

\begin{remark}\label{rem:NONSYM}
We stress that the proof of Proposition \ref{lem:Linfw} uses duality and so the symmetry of the SIPG method is essential to it; it would break down for any of the nonsymmetric versions of the IP family. In particular by considering a nonsymmetric method (either NIPG or IIPG), the last term in \eqref{eq:NONSYM} (which is negative and has the correct sign) would be either positive (for NIPG) or zero (for IIPG). 
\par \noindent
More precisely, if we had defined $\tilde{b}_h(v,w)$  by
\begin{align*}
\tilde{b}_h(v,w)=-\sum_{e\in\cE_h} \int_e \smean{\nabla \Pi_h(v)}\sjump{w} ds   +\theta \sum_{e\in\cE_h} \int_e \smean{\nabla \Pi_h w}\sjump{v}\,ds +\sum_{e\in\cE_h} \int_e
\frac{\gamma}{h_e}\sjump{v}\sjump{w}\,ds,
\end{align*}
with $\theta=-1, 1 ~\text{and} ~0$ corresponding  to SIPG, NIPG and  IIPG methods, respectively, then while estimating the term $\widetilde{\cA_h}(\xi^1,\mathcal{G})$ in equation \eqref{eq:NONSYM} in the proof of Proposition $\ref{lem:Linfw}$, we would have got
\begin{align*}
\widetilde{\cA}_h(\xi^{1},\mathcal{G}) &= a_h(\xi^1,\mathcal{G})+\tilde{b}_h((\xi^1,\mathcal{G})
= (\nabla_h \xi^1, \nabla \mathcal{G}) +\theta \sum_{e\in\cE_h } \int_e \smean{\nabla \Pi_h(\mathcal{G})}\sjump{\xi^1} ds && \notag\\
 &=( \xi^1, -\Delta \mathcal{G}) +\sum_{e\in \Eh} \int_e\sjump{\xi^1}\smean{\nabla \mathcal{G}}\,ds +\theta \sum_{e\in\cE_h} \int_e \smean{\nabla \Pi_h(\mathcal{G})}\sjump{\xi^1} ds.
\end{align*}
\par \noindent
This would preclude any further analysis for $\theta=0,1$ in the sense that optimal estimates can not be obtained for these choices of $\theta$. 
We further notice, that a naive analysis without accounting for the nonconformity of the method will not reveal this issue (or show this drawback), but one would be  inevitably providing only an estimator for the conforming part of the solution.
\end{remark}

\subsection{Proof of Reliability Theorem \ref{thm:reliability}}

We have now all ingredients to complete the proof of Theorem \ref{thm:reliability}.
\par
\noindent
\textbf{Proof of Theorem \ref{thm:reliability}:}  We have,
\begin{align*}
\|u-u_h\|_{L^{\infty}(\Omega)}&\lesssim  \|u-u_h^{conf}\|_{L^{\infty}(\Omega)}+\|u_h^{conf}-u_h\|_{L^{\infty}(\Omega)}.
\end{align*}
From Lemma \ref{lem:barrier} we find,
\begin{equation*}
w-\|w\|_{L^{\infty}(\Omega)}-\|({u_h^{conf}}-\chi)^{+}\|_{L^{\infty}(\{\sigma_h<0\})} \lesssim  u-u_h^{conf}  \lesssim w+\|w\|_{L^{\infty}(\Omega)}+\|(\chi-u_h^{conf})^{+}\|_{L^{\infty}(\Omega)},
\end{equation*}
which implies,
\begin{equation}\label{eq:Eqn2}
\|u-u_h^{conf}\|_{L^{\infty}(\Omega)} \lesssim 2 \|w\|_{L^{\infty}(\Omega)}+\|(\chi-u_h^{conf})^{+}\|_{L^{\infty}(\Omega)}+\|({u_h^{conf}}-\chi)^{+}\|_{L^{\infty}(\{\sigma_h<0\})}.
\end{equation}
{
By combining now \eqref{eq:Eqn2} and \eqref{eq:apen3} together with the observation that
\begin{align*}
\|(\chi-u_h^{conf})^{+}\|_{L^{\infty}(\Omega)} &\lesssim  \|(\chi-u_h)^{+}\|_{L^{\infty}(\Omega)} + \|u_h-u_h^{conf}\|_{L^{\infty}(\cT_h)} \\
& \lesssim  \|(\chi-u_h)^{+}\|_{L^{\infty}(\Omega)}+\|\sjump{u_h}\|_{L^{\infty}(\cE_h)},
\end{align*}
and 
\begin{align*}
\|({u_h^{conf}}-\chi)^{+}\|_{L^{\infty}(\{\sigma_h<0\})} & \lesssim \|({u_h^{conf}}-u_h)\|_{L^{\infty}(\{\sigma_h<0\})} + \|({u_h}-\chi)^{+}\|_{L^{\infty}(\{\sigma_h<0\})} \\
& \lesssim \|({u_h}-\chi)^{+}\|_{L^{\infty}(\{\sigma_h<0\})}+\|\sjump{u_h}\|_{L^{\infty}(\cE_h)},
\end{align*}
}
 we obtain
\begin{align*}
\|u-{u_h}\|_{L^{\infty}(\Omega)} \lesssim  \|w\|_{L^{\infty}(\Omega)}+\|(\chi-u_h)^{+}\|_{L^{\infty}(\Omega)}+\|({u_h}-\chi)^{+}\|_{L^{\infty}(\{\sigma_h<0\})}+\|\sjump{u_h}\|_{L^{\infty}(\cE_h)} .
\end{align*}
Using now the estimates for $ \|w\|_{L^{\infty}(\Omega)}$ from Proposition \ref{lem:Linfw}, the result follows.\\

\par\noindent
We remark here that, a use of Lemma \ref{lem:Estmsig} leads to the following reliability estimates:
\begin{corollary}\label{cor:Rel} Let $u \in \cK$ and $u_h \in \cK_h$ be the  solution of \eqref{eq:MP} and \eqref{eq:Ah}, respectively. Then,
\begin{align*}
\|u-{u_h}\|_{L^{\infty}(\Omega)}\lesssim \tilde{\eta}_h,
\end{align*}
where,
\begin{align*}
\tilde{\eta}_h&=  |\log h_{min}| \Big( \eta_1 +\eta_2 +\eta_3 \Big)+\|(\chi-u_h)^{+}\|_{L^{\infty}(\Omega)}+\|({u_h}-\chi)^{+}\|_{L^{\infty}(\{\sigma_h<0\})} + \max_{T \in \cT_h}Osc(f,T).
\end{align*}
\end{corollary}
{ 
\begin{remark}\label{rmk:r1} It can be observed directly from the definitions \eqref{conth}, \eqref{nconth}, \eqref{freebh} of discrete contact, noncontact and free boundary sets that $\{T \in \cT_h: \sigma_h<0~~\text{on}~ T\}) \subset \mathbb{C}_h \cup \mathbb{M}_h$. Therefore the term $\|({u_h}-\chi)^{+}\|_{L^{\infty}(\{\sigma_h<0\})}$ measure the error in approximation of the obstacle in the discrete contact set and discrete free boundary set and  detect the non-affine situation $\chi \neq \chi_h$.

The term $\|(\chi-u_h)^{+}\|_{L^{\infty}(\Omega)}$ measure the error in the violation of the obstacle constraint and allow to detect when $\chi> u_h$.
\end{remark}
}
\subsection{Estimation of the error in the Lagrange multiplier}
\par\noindent
We derive a bound for the error in the Lagrange multiplier in a dual norm in terms of the error estimator.
Let $D\subseteq\Omega$ be any set. Denote $\cT_D$ by the set of all $T\in \cT_h$ such that $T\subset D$. Furthermore, let $\cE_D$ be the set of all edges/faces in $\cE_h$ that are in $D$.
\par\noindent
We introduce the functional space
\begin{equation}\label{defW}
\mathcal{W}_D:= W^{2,1}_{0}(D)= \{ v \in W^{2,1}(D)\,\,: \,\,  v=0  ~\mbox{and }~ \nabla v\cdot {\bf n} =0 \mbox{  on  } \partial D\},
\end{equation}
{
endowed with the norm $\|v\|_{\mathcal{W_D}} :=|v|_{W^{1,1}(D)}+ |v|_{W^{2,1}(D)}$. For notational convenience, we set $\mathcal{W}:=\mathcal{W}_\O$ which is essentially the space $W^{2,1}_{0}(\O)$.\\
}
\par \noindent
For any $\mathcal{F} \in \mathcal{W}_D^{\ast}$, the norm $\|\mathcal{F}\|_{-2,\infty,D}$ is defined by
\begin{align*}
\|\mathcal{F} \|_{-2,\infty,D} = \|\mathcal{F} \|_{\mathcal{W}_D^{\ast}} := \sup \{ \langle \mathcal{F} ,v \rangle~ :~  \tilde v \in \mathcal{W}_D \,\, ,\,\, |\tilde v|_{\mathcal{W}_D} \leq 1 \}.
\end{align*}
\par \noindent
The subsequent analysis requires the bound on $\|G_h\|_{-2,\infty,\Omega}$ which is estimated in next lemma.
\begin{lemma} \label{eq:Grel} It holds that
\begin{align*}
\|G_h\|_{{-2,\infty,\O}} \lesssim \eta_1 +\eta_2 +\eta_3 +\eta_4.
\end{align*}

\end{lemma}

\begin{proof}

\par \noindent
For any $ v \in \mathcal{W}$ and $v_h=\Pi_hv$, using the definition of $G_h$, \eqref{eq:extend:sigmadef} and \eqref{def-sigmah}, we have,
\begin{align}\label{eq:Lrel}
\langle G_h, v \rangle &=\widetilde{\cA}_h(u-u_h,v)+\langle {\tilde\sigma}(u)-{\sigma_h},v \rangle \notag \\
&= (f,v)-\widetilde{\cA}_h(u_h,v)  -({\sigma_h},v ) \notag\\
&=(f,v-v_h)-\widetilde{\cA}_h(u_h,v-v_h)  -({\sigma_h},v-v_h ) \notag\\
&\qquad +(f,v_h)-\widetilde{\cA}_h(u_h,v_h)  -({\sigma_h},v_h ) \notag\\
&=(f-\sigma_h,v-v_h)-\widetilde{\cA}_h(u_h,v-v_h)+\langle{\sigma_h},v_h \rangle_h-({\sigma_h},v_h ).
\end{align}
\par \noindent
We begin by estimating the terms on the right hand side of  \eqref{eq:Lrel} as follows: a use of H{\"o}lders' inequality and Lemma \ref{lem:approx_proj} yields,
\begin{align}\label{eq:Grele1}
(f-\sigma_h,v-v_h) &\lesssim \sum_{T \in \cT_h} \|h_T^2(f-\sigma_h)\|_{L^{\infty}(T)} \|h_T^{-2}(v-v_h)\|_{L^1(T)}\notag\\
& \lesssim \max_{T \in \cT_h}\|h_T^2(f-\sigma_h)\|_{L^{\infty}(T)}\sum_{T \in \cT_h} \|h_T^{-2}(v-v_h)\|_{L^1(T)}\notag\\
& \lesssim \eta_1 |v|_{W^{2,1}(\O)}.
\end{align}
To estimate the second term of \eqref{eq:Lrel}, we use integration by parts taking into account that $u_h \in \mathbb{P}^{1}(\Th)$,
\begin{align*}
\widetilde{\cA}_h(u_h,v-v_h)&=a_h(u_h,v-v_h)+\tilde{b}_h(u_h,v-v_h) &&\\
& =\sum_{T\in\cT_h} \int_{\partial T}\frac{\partial u_h|_T}{\partial n_T}(v-v_h)\,ds+\tilde{b}_h(u_h,v-v_h) &&\\
& =\sum_{e\in \Eho} \int_e\sjump{\nabla u_h}\smean{v-v_h}\,ds +\sum_{e\in\cE_h} \int_e
\frac{\gamma}{h_e}\sjump{u_h}\sjump{v-v_h}\,ds, &&
\end{align*}
where in the last step, we used the fact that $\Pi_h$ is a projection (hence $\Pi_h(v-\Pi_h v)=0$). Therefore, using H{\"o}lders' inequality, trace inequality \eqref{eq:traceAg} and Lemma \ref{lem:approx_proj} we find
\begin{align}\label{eq:Grele2}
|\widetilde{\cA}_h(u_h,v-v_h)| &\lesssim \sum_{e \in \Eho}  \| h_e \sjump{\nabla u_h} \|_{L^{\infty}(e)} h_e^{-1} \|\smean{v-v_h}\|_{L^{1}(e)} 
+ \|\sjump{u_h} \|_{L^{\infty}(\Eh)}\sum_{e \in \Eh} h_e^{-1}\|\sjump{v-v_h}\|_{L^{1}(e)} &&\notag\\
&\lesssim  \left( \max_{e \in \Eho}  \| h_e \sjump{\nabla u_h} \|_{L^{\infty}(e)} + \|\sjump{u_h} \|_{L^{\infty}(\Eh)}\right)\notag \\ & \quad\quad \Big (\sum_{T\in \Th}  \left(  h_T^{-2}\|v-v_h\|_{L^{1}(T)} +h_T^{-1}\|\nabla( v-v_h)\|_{L^{1}(T)} \right) \Big)\notag\\
&\lesssim (\eta_2 +\eta_3)|v|_{W^{2,1}(\O)}.
\end{align}
We estimate the last two terms of \eqref{eq:Lrel} using Lemma \ref{lem:Ihapprox},  inverse  inequality  \eqref{eq:inverse} together with the stability estimate \eqref{stabPik0} from Lemma \ref{lem:approx_proj} as follows,
\begin{align}\label{eq:Grele3}
\langle{\sigma_h},v_h \rangle_h-({\sigma_h},v_h )&\lesssim   \sum_{T \in \Th^C }\|  h_T^2 \nabla \sigma_h\|_{L^{d}(T)} \|\nabla v_h\|_{L^{d/(d-1)}(T)} \notag \\ 
&\lesssim  \sum_{T \in \Th^C }\|  h_T^2 \nabla \sigma_h\|_{L^{d}(T)} \|\nabla v_h\|_{W^{1,1}(T)} \notag \\
&\lesssim \left(\max_{T \in \Th^C }\|  h_T^2 \nabla \sigma_h\|_{L^{d}(T)}  \right)  \sum_{T \in \Th^C }\|\nabla v\|_{W^{1,1}(T)} \notag\\
& \leq \eta_4 |v|_{W^{2,1}(\O)}.
\end{align}

Combining \eqref{eq:Grele1}, \eqref{eq:Grele2} and \eqref{eq:Grele3} together with \eqref{eq:Lrel}, we obtain
\begin{align*}
\|G_h\|_{{-2,\infty,\O}} \lesssim \eta_1 +\eta_2 +\eta_3 +\eta_4.
\end{align*}
\end{proof}

\par \noindent
Now, we show that the following bound on the error $\|{\tilde\sigma}(u)-\sigma_h\|_{{-2,\infty,\O}}$ holds.
\begin{proposition} Let $\tilde\sigma(u)$ and $\sigma_h$ be as defined in \eqref{eq:extend:sigmadef} and \eqref{def-sigmah}, respectively. Then,
\begin{align*}
\|{\tilde\sigma}(u)-\sigma_h\|_{{-2,\infty,\O}} \lesssim \eta_h.
\end{align*}
\end{proposition}
\begin{proof} For any $ v \in \mathcal{W}$, from equation \eqref{eq:Gh}, we have
\begin{align}\label{eq:Lerr1}
\langle {\tilde\sigma}(u)-{\sigma_h},v \rangle &=\langle G_h, v \rangle-\widetilde{\cA}_h(u-u_h,v).
\end{align}
Now, a use of integration by parts and \eqref{ext:b} yields
\begin{align}\label{eq:Lerr}
\widetilde{\cA}_h(u-u_h,v)&=\sum_{T\in\cT_h} \int_{ T}\nabla (u-u_h)\cdot \nabla v\,dx+\tilde{b}_h(u-u_h,v) \notag\\
&=-\sum_{T\in\cT_h} \int_{ T} (u-u_h) \Delta v\,dx+\sum_{T\in\cT_h} \int_{\partial T}(u-u_h)\frac{\partial v}{\partial n_T}\,ds+\tilde{b}_h(u-u_h,v) \notag\\
&=-\sum_{T\in\cT_h} \int_{ T} (u-u_h) \Delta v\,dx-\sum_{e\in\cE_h} \int_e \sjump{u_h} \smean{\nabla (v-\Pi_hv)}\,ds.
\end{align}
Therefore, in view of trace inequality \eqref{eq:traceAg} and Lemma \ref{lem:approx_proj} we obtain
\begin{align}\label{eq:Lerr2}
|\widetilde{\cA}_h(u-u_h,v)| \lesssim  \left (\|u-u_h\|_{L^{\infty}(\O)} +\|\sjump{u_h}\|_{L^{\infty}(\cE_h)}\right) |v|_{W^{2,1}(\O)}.
\end{align}
Combining \eqref{eq:Lerr2} together with \eqref{eq:Lerr1}, we have
\begin{align}\label{eq:Lbound}
\|{\tilde\sigma}(u)-\sigma_h\|_{{-2,\infty,\O}} \lesssim \|u-u_h\|_{L^{\infty}(\O)}+\|\sjump{u_h}\|_{L^{\infty}(\cE_h)}+\|G_h\|_{{-2,\infty,\O}}.
\end{align}
Finally, in view of \eqref{eq:Lbound}, Lemma \ref{eq:Grel} and Theorem \ref{thm:reliability} we have the desired reliability estimate of $\eta_h$ for the error in Lagrange multiplier.

\end{proof}

\subsection{Error Estimator for Conforming Finite Element Method}
The conforming finite element method for the model problem \eqref{eq:MP} is to find $\tilde{u}_h \in \cK_h^{conf}$ such that
\begin{align}\label{eq:DCP}
a(\tilde{u}_h,v_h-u_h) \geq (f,v_h-u_h) \quad \forall v_h \in \cK_h^{conf},
\end{align}
where
\begin{align*}
\cK_h^{conf}:= \{v_h\in V_h^{conf}: v_h(p)\geq \chi_h(p),\;\; \forall p\in \cV_h^i\}.
\end{align*}
In this case, our analysis will lead to the following reliability estimates:
\begin{corollary}\label{thm:Creliability} Let $u \in \cK$ and $\tilde{u}_h \in \cK_h^{conf}$ be the  solution of \eqref{eq:MP} and \eqref{eq:DCP}, respectively. Then,
\begin{align*}
\|u-{\tilde{u}_h}\|_{L^{\infty}(\Omega)}\lesssim  \zeta_h,
\end{align*}
where
\begin{align*}
{\zeta}_h&=  |\log h_{min}| \Big( \eta_1 +\eta_2\Big)+\|(\chi-\tilde{u}_h)^{+}\|_{L^{\infty}(\Omega)}+\|({\tilde{u}_h}-\chi)^{+}\|_{L^{\infty}(\{\sigma_h<0\})} + \max_{T \in \cT_h}Osc(f,T).
\end{align*}
\end{corollary}

\begin{remark} This error estimator $\zeta_h$ is comparable with the error estimator reported in \cite{NSV:2002:Ptwise}. The slight improvement in the log term is obtained using the refined Green's function estimates.
\end{remark}

\section{Efficiency of the error estimator}\label{Sec:eff}

\par \noindent
In this section, we derive the local efficiency estimates of a posteriori error estimator $\eta_h$ obtained in the previous section. 
Therein, we require a bound on a negative norm of the residual in terms of the $L^{\infty}$ error in the solution $u$ and a negative norm of the error in the Lagrange multiplier.
Let $D\subset\Omega$ be a set formed by some triangles in $\cT_h$.  Essentially in what follows, $D$ is either a simplex or union of two simplices sharing a face/edge. Denote $\cT_D$ by the set of all $T\in \cT_h$ such that $T\subset D$. Furthermore, let $\cE_D$ be the set of all edges/faces in $\cE_h$ that are in $D$.

\par \noindent
Note that for any $\tilde v \in \mathcal{W}_D$ (defined in \eqref{defW}) extended to $\Omega$ by zero outside of $D$, the resulting extended function denoted by $v$ is in $\mathcal{W}=W^{2,1}_{0}(\O)$.
For any $\tilde v \in \mathcal{W}_D$, using the definition of $G_h$ we have,

\begin{align}\label{eq:Ghest}
\langle G_h,\tilde v \rangle&=\langle G_h, v \rangle =\widetilde{\cA}_h(u-u_h,v)+\langle {\tilde\sigma}(u)-{\sigma_h},v \rangle.
\end{align}
Therein, we simplify the first term as in \eqref{eq:Lerr} to get
\begin{align*}
\widetilde{\cA}_h(u-u_h,v)=- \sum_{T \in \cT_h}\int_{T}(u-u_h) \Delta v~dx-\sum_{e\in\cE_h} \int_e \sjump{{u_h}} \smean{\nabla (v-\Pi_h v)}\,ds.
\end{align*}

Therefore,
\begin{align}
\langle G_h,\tilde v \rangle &=- \int_{\O}(u-u_h) \Delta v~dx-\sum_{e\in\cE_h} \int_e \sjump{{u_h}} \smean{\nabla (v-\Pi_h v)}\,ds +\langle {\tilde\sigma}(u)-{\sigma_h},v \rangle \nonumber\\
&=- \int_{D}(u-u_h) \Delta \tilde v~dx-\sum_{e\in\cE_D} \int_e\sjump{{u_h}} \smean{\nabla (\tilde v-\Pi_h \tilde v)}\,ds +\langle {\tilde\sigma}(u)-{\sigma_h},\tilde v \rangle. \label{eq:Ghest1}
\end{align}

We begin by estimating the terms on the right hand side of  \eqref{eq:Ghest1} as follows: first,
\begin{align*}
\left|\int_{D}(u-u_h) \Delta v~dx\right| \lesssim \|u-u_h\|_{L^{\infty}(D)}  | v|_{\mathcal{W}},
\end{align*}
 a use of H{\"o}lders' inequality, trace inequality \eqref{eq:traceAg} and Lemma \ref{lem:approx_proj} yields,
\begin{align*}
\sum_{e\in\cE_D} \int_e \sjump{{u_h}}\smean{\nabla (v-\Pi_h v)}\,ds &\leq  \sum_{e\in\cE_D} \|\smean{\nabla (v-\Pi_h v)} \|_{L^1(e)} \|\sjump{{u_h}}\|_{L^{\infty}(e)} \\
& \lesssim \|\sjump{u_h}\|_{L^{\infty}(\cE_D)} \sum_{T \in \cT_D}\Big( h_T^{-1} \|{\nabla (v-\Pi_h v)} \|_{L^1(T)} + |v|_{W^{2,1}(T)} \Big)\\
& \lesssim \|\sjump{u_h}\|_{L^{\infty}(\cE_D)} \sum_{T \in \cT_D} |v|_{W^{2,1}(T)} \lesssim \|\sjump{u_h}\|_{L^{\infty}(\cE_D)}  | \tilde v|_{\mathcal{W}_D},
\end{align*}
and,
\begin{align*}
\langle {\tilde\sigma}(u)-{\sigma_h},v \rangle \lesssim  \|{\tilde\sigma}(u)-\sigma_h\|_{{-2,\infty,D}} | \tilde v|_{\mathcal{W}_D}.
\end{align*}
\par \noindent
Now, the following bound on the Galerkin functional $G_h$, in the dual space of $\mathcal{W}_D$, is then immediate from equation \eqref{eq:Ghest1}:
\begin{align}\label{eq:Ghbound}
\|G_h\|_{{-2,\infty,D}} \lesssim \|u-u_h\|_{L^{\infty}(D)}+\|\sjump{u_h}\|_{L^{\infty}(\cE_D)}+\|{\tilde\sigma}(u)-\sigma_h\|_{{-2,\infty,D}}.
\end{align}
This estimate on $G_h$ will be useful in deriving the local efficiency estimates. 

\par
\noindent
The derivation of the following local efficiency estimates of the error estimator uses the standard bubble functions technique \cite[Chapter 2, page 23]{AO:2000:Book}. We have discussed the main ideas involved in proving these efficiency estimates and skipped the standard details.

\begin{theorem}
Let $u \in \cK$ and $u_h \in \cK_h$ be the solutions of \eqref{eq:MP} and \eqref{eq:Ah}, respectively. Then it holds that,
\begin{align}
h_T^2\|f-\sigma_h\|_{L^{\infty}(T)}+\|(\chi-u_h)^{+}\|_{L^{\infty}(T)} &\lesssim  \|u-u_h\|_{L^{\infty}(T)}+ \|\sigma(u)-\sigma_h\|_{-2,\infty,T}  \label{eq:Eff1} \\ & \quad +Osc(f,T)\quad \forall T \in \cT_h, \notag
\end{align}
\begin{align}
\|\sjump{u_h}\|_{L^{\infty}(e)} &\lesssim \|u-u_h\|_{L^{\infty}(\omega_e)} \quad \forall e\in \Eh, \label{eq:EffJump}\\
h_e\|\sjump{\nabla u_h}\|_{L^{\infty}(e)}&\lesssim \|u-u_h\|_{L^{\infty}(\omega_e)}+\|\sigma(u)-\sigma_h\|_{-2,\infty,\omega_e}+Osc(f,\omega_e) \qquad \forall e\in \Eho, \label{eq:Eff2}\\
h_T^2\|\nabla \sigma_h\|_{L^{d}(T)}&\lesssim\|u-u_h\|_{L^{\infty}(T)}+\|\sigma(u)-\sigma_h\|_{-2,\infty,T} +Osc(f,T)\qquad \forall ~T\in \mathbb{C}_h \cup \mathbb{M}_h\;.\label{eq:Eff3} 
\end{align}
{ Further, assuming  $\chi<0$ on $T \cap \partial \Omega$  if $ T \in \mathcal{M}_h$ and $T\cap \partial \Omega \neq \emptyset$, it holds for all elements $T \in \cT_h$ such that $T \cap \{\sigma_h <0\} \neq \emptyset$ we have}
\begin{align}
\|( u_h-\chi)^+\|_{L^{\infty}(T)}&\lesssim \|u-u_h\|_{L^{\infty}(T)}+\|\sigma(u)-\sigma_h\|_{-2,\infty,T}  +\|(\chi-\chi_h)^+\|_{L^{\infty}(T)} \label{eq:Eff4}\\ & +\sum_{p\in \cV_T, T \in \mathbb{M}_h}\sum_{e\in \cE_p\cap\Eho}h_e(\|\sjump{\nabla u_h}\|_{L^{\infty}(e)}+\|\sjump{\nabla\chi_h}\|_{L^{\infty}(e)}).\notag
\end{align}
\par \noindent
Here, $\omega_e$ denotes the union of elements sharing the face $e$ and $Osc(f,T)$ is  defined in \eqref{def:Osc}.
\end{theorem}

{
\begin{remark} The proof of efficiency estimate \eqref{eq:Eff3} follows directly from Lemma \ref{lem:Estmsig} and the estimate \eqref{eq:Eff1}.
\end{remark}
}
\begin{proof}

${\color{blue}\bullet}$  We start with showing that for any $T \in \cT_h$,
\begin{align*}
h_T^2\|f-\sigma_h\|_{L^{\infty}(T)}+\|(\chi-u_h)^{+}\|_{L^{\infty}(T)} &\lesssim \|u-u_h\|_{L^{\infty}(T)}+ \|\sigma(u)-\sigma_h\|_{-2,\infty,T}+Osc(f,T).
\end{align*}
\noindent
Note that, the bound on the second term of left hand side in above estimate  follows immediately since $(\chi-u_h)^{+} \leq (u-u_h)^{+}$ in view of  $u \geq \chi$. We have,
\begin{align}\label{eq:effeq1}
\|(\chi-u_h)^{+}\|_{L^{\infty}(T)} \lesssim  \|u-u_h\|_{L^{\infty}(T)}.
\end{align}
To estimate $h_T^2\|f-\sigma_h\|_{L^{\infty}(T)}$, let $b_T \in \mathbb{P}^{2(d+1)}(T)\cap H^2_0(T)$ be the polynomial bubble function which vanishes up to first order derivative on $\partial T$, attains unity at the barycenter of T and $\|b_T\|_{L^{\infty}(T)}=1$.
\par
\noindent
Let $\tilde{f} \in \mathbb{P}^0(T)$ be arbitrary. Set $\phi_T=b_T (\tilde{f}-\sigma_h)$ on $T$ and let $\phi $ be the extension of $\phi_T$ to $\bar\Omega$ by zero. It is easy to see that $\phi$ is continuously differentiable on $\bar\Omega$ and $\phi \in H^2_0(\Omega)$.
\noindent
Using the equivalence of norms in finite dimensional normed spaces, we have the existence of positive constants $C_1$ and $C_2$ such that
\begin{align*}
C_1 \int_T (\tilde{f}-\sigma_h) \phi_T ~dx\leq \|\tilde{f}-\sigma_h\|_{L^{2}(T)}^2 \leq C_2 \int_T (\tilde{f}-\sigma_h) \phi_T~dx.
\end{align*}
\noindent
Taking into consideration that $u_h$ is linear on T, a use of \eqref{eq:Gh2}  and integration by parts yielding $a_h(u_h,\phi)=0$, we have
\begin{align}\label{eq:eff3}
\|\tilde{f}-\sigma_h\|_{L^{2}(T)}^2 &\lesssim \int_T (\tilde{f}-\sigma_h) \phi_T~dx 
=\int_T (f-\sigma_h) \phi_T~dx +\int_T (\tilde{f}-f) \phi_T ~dx\notag\\
&\lesssim a(u,\phi) + \langle \sigma(u) -\sigma_h, \phi \rangle+\int_T (\tilde{f}-f) \phi_T ~dx\notag\\
&\lesssim \langle G_h,\phi \rangle+\tilde{b}_h(u_h,\phi) +\int_T (\tilde{f}-f) \phi_T ~dx.
\end{align}
\noindent
Now, using the inverse estimate \eqref{eq:inverse2}, we have,
\begin{align}
h_T^2\|\tilde{f}-\sigma_h\|_{L^{\infty}(T)} \lesssim  h_T^{2-\frac{d}{2}} \|\tilde{f}-\sigma_h\|_{L^{2}(T)}. \label{eq:ef2}
\end{align}
Therefore, from equations \eqref{eq:ef2} and \eqref{eq:eff3}, we have
\begin{align}\label{eq:Eeq1}
h_T^4\|\tilde{f}-\sigma_h\|_{L^{\infty}(T)}^2 &\lesssim h_T^{4-d} \|\tilde{f}-\sigma_h\|_{L^{2}(T)}^2 \notag\\
&\lesssim h_T^{4-d} \Big(\langle G_h,\phi \rangle+\tilde{b}_h(u_h,\phi) +\int_T (\tilde{f}-f) \phi_T ~dx \Big).
\end{align}
To bound the terms of the right hand side of the last equation, we first estimate the term $\tilde{b}_h(u_h,\phi)$ as follows:  a use of H\"older inequality, trace inequality \eqref{eq:traceAg}, inverse inequality and Lemma \ref{lem:approx_proj} yields,
\begin{align} \label{eq:bhNb}
\tilde{b}_h(u_h,\phi) &=- \sum_{e \in \partial T} \int_e \sjump{u_h} \smean{\nabla \Pi_h \phi}~ds 
\leq \sum_{e \in \partial T}  \|\sjump{u_h}\|_{L^{\infty}(e)} \|\smean{\nabla \Pi_h \phi}\|_{L^1(e)} \notag\\
& \lesssim \|\sjump{u_h}\|_{L^{\infty}(\partial T)} (h_T^{-1} \|\nabla \Pi_h \phi\|_{L^1(T)} + |\nabla \Pi_h \phi|_{W^{1,1}(T)} ) \notag\\
& \lesssim \|\sjump{u_h}\|_{L^{\infty}(\partial T)} (h_T^{-2} \| \Pi_h \phi\|_{L^1(T)} + |\nabla \Pi_h \phi|_{W^{1,1}(T)} ) \notag\\
& \lesssim \|\sjump{u_h}\|_{L^{\infty}(\partial T)} (h_T^{-2} \| \phi_T\|_{L^1(T)} + |\phi_T |_{W^{2,1}(T)} ) \notag \\
& \lesssim h_T^{-2}  \|\sjump{u_h}\|_{L^{\infty}(\partial T)} \| \phi_T\|_{L^1(T)}.
\end{align}
Combining \eqref{eq:Eeq1} and \eqref{eq:bhNb}, we have
\begin{align}\label{eq:Eeq11}
h_T^4\|\tilde{f}-\sigma_h\|_{L^{\infty}(T)}^2 &\lesssim h_T^{4-d} \Big( \|G_h\|_{-2,\infty,T} \|\phi_T\|_{W^{2,1}(T)}+ h_T^{-2} \|\sjump{u_h}\|_{L^{\infty}(\partial T)}   \| \phi_T\|_{L^1(T)} \notag\\\qquad\qquad& + \|\tilde{f}-f\|_{L^{\infty}(T)} \|\phi_T\|_{L^1(T)} \Big ).
\end{align}
Further, using inverse estimates \eqref{eq:inverse}, \eqref{eq:inverse1}, H\"older inequality and the structure of $b_T$, we have
\begin{align*}
|\phi_T|_{W^{2,1}(T)} &\lesssim h_T^{-2}\|\phi_T\|_{L^1(T)} 
 \lesssim h_T^{d-2} \|\phi_T\|_{L^{\infty}(T)} 
\lesssim  h_T^{d-2} \|\tilde{f}-\sigma_h\|_{L^{\infty}(T)}
\end{align*}
and
\begin{align*}
\|\phi_T\|_{L^1(T)} &\lesssim  h_T^{d} \|\phi_T\|_{L^{\infty}(T)} 
\lesssim C h_T^{d} \|\tilde{f}-\sigma_h\|_{L^{\infty}(T)}.
\end{align*}
Therefore, from  \eqref{eq:Eeq11}, we have
\begin{align*}
h_T^4\|\tilde{f}-\sigma_h\|_{L^{\infty}(T)}^2 &\lesssim h_T^{4-d}  h_T^{d-2} \Big( \|G_h\|_{-2,\infty,T} + \|\sjump{u_h}\|_{L^{\infty}(\partial T)} + h_T^2\|\tilde{f}-f\|_{L^{\infty}(T)}  \Big ) \|\tilde{f}-\sigma_h\|_{L^{\infty}(T)}.
\end{align*}
Finally, from \eqref{eq:Ghbound}, we obtain
\begin{align}\label{eq:effeq2}
h_T^2\|\tilde{f}-\sigma_h\|_{L^{\infty}(T)} \lesssim \|u-u_h\|_{L^{\infty}(T)}+\|\sjump{u_h}\|_{L^{\infty}(\partial T)}+ \|\sigma(u)-\sigma_h\|_{-2,\infty,T}+Osc(f,T).
\end{align}
Combining estimates \eqref{eq:effeq1} and \eqref{eq:effeq2} taking into account \eqref{eq:EffJump}, we obtain the estimate \eqref{eq:Eff1}.\\

${\color{blue}\bullet}$ The estimate  \eqref{eq:EffJump} can be obtained easily. Let $e \in \Eho$ and $\omega_e=\overline{T}_+ \cup \overline{T}_-$, where $T_+$ and $T_-$ are elements sharing the face $e$. Then, by taking into consideration $\sjump{u}=0$ on $e$, we have
\begin{align*} 
\|\sjump{u_h}\|_{L^{\infty}(e)}& = \|\sjump{u-u_h}\|_{L^{\infty}(e)}
\lesssim \|u-u_h\|_{L^{\infty}(\omega_e)}.
\end{align*}
\par
\noindent
For $e \in \Ehb$, similar arguments can be used to estimate $\|\sjump{u_h}\|_{L^{\infty}(e)} $ taking into account $u=0$ on $\partial \Omega$.\\

${\color{blue}\bullet}$
Next, we provide the proof of the estimate \eqref{eq:Eff2}.
Let $e \in \Eho$ and $\omega_e=\overline{T}_+ \cup \overline{T}_-$, where $T_+$ and $T_-$ are elements sharing the face e. Let $b_e\in \mathbb{P}^{4d}(\omega_e)\cap H^2_0(\omega_e)$ denote the edge bubble function, which takes unit value at the center of $e$ and $\|b_e\|_{L^{\infty}(\omega_e)}=1$.
\par
\noindent
Define $\phi_e= \sjump{\nabla u_h} b_e$ on $\omega_e$ and let $\phi$ be the extension of $\phi_e$ by zero to $\bar \Omega$, clearly $\phi\in H^2_0(\Omega)$. 
Using the inverse estimate \eqref{eq:inverse3}, we have
\begin{align} \label{eq:eq1}
h_e\|\sjump{\nabla u_h}\|_{L^{\infty}(e)}&\lesssim h_e^{\frac{3-d}{2}} \|\sjump{\nabla u_h}\|_{L^{2}(e)}.
\end{align}
To bound the right hand side of the above estimate, a use of equivalence of norms in finite dimension, integration by parts and \eqref{eq:Gh2} yields,
\begin{align*}
\|\sjump{\nabla u_h}\|_{L^{2}(e)}^2&\lesssim \int_e \sjump{\nabla u_h} \phi_e ~ds=\int_{\omega_e} \nabla u_h \cdot \nabla \phi_e~ds
 =a_h(u_h,\phi)\\
&\lesssim (f-\sigma_h,\phi)-\langle G_h,\phi \rangle -\tilde{b}_h(u_h,\phi).
\end{align*}
\noindent
Note that, as in \eqref{eq:bhNb}, we have the following estimate of the term $\tilde{b}_h(u_h,\phi)$ :
\begin{align}\label{eq:bnNb1}
\tilde{b}_h(u_h,\phi) \lesssim \|\sjump{u_h}\|_{L^{\infty}(\partial \omega_e)}\Big( \sum_{T \in \omega_e} h_T^{-2}\| \phi_e\|_{L^1(T)} \Big).
\end{align}
Therefore, using equations \eqref{eq:eq1}, H\"older inequality, inverse estimates \eqref{eq:inverse} and the structure of the bubble function $b_e$, we have
\begin{align*}
h_e^2\|\sjump{\nabla u_h}\|_{L^{\infty}(e)}^2&\lesssim h_e^{{3-d}} \Big((f-\sigma_h,\phi)-\langle G_h,\phi \rangle -\tilde{b}_h(u_h,\phi) \Big)\\
& \lesssim h_e^{{3-d}} \Big( \|{f}-\sigma_h\|_{L^{\infty}(\omega_e)} \|\phi_e\|_{L^1(\omega_e)} + \|G_h\|_{-2,\infty,\omega_e} |\phi_e|_{W^{2,1}(\omega_e)} \\ &\qquad + h_e^{-2}\|\sjump{u_h}\|_{L^{\infty}(\partial \omega_e)} \| \phi_e\|_{L^1(\omega_e)} \Big)\\
& \lesssim h_e^{{3-d}} \Big( h_e^d \|{f}-\sigma_h\|_{L^{\infty}(\omega_e)} \|\phi_e\|_{L^{\infty}(\omega_e)} + h_e^{d-2}\|G_h\|_{-2,\infty,\omega_e} \|\phi_e\|_{L^{\infty}(\omega_e)} \\ &\qquad + h_e^{d-2} \|\sjump{u_h}\|_{L^{\infty}(\partial \omega_e)} \|\phi_e\|_{L^{\infty}(\omega_e)} \Big)\\
& \lesssim h_e \Big( h_e^2 \|{f}-\sigma_h\|_{L^{\infty}(\omega_e)}  + \|G_h\|_{-2,\infty,\omega_e}  +  \|\sjump{u_h}\|_{L^{\infty}(\partial \omega_e)}  \Big) \|\sjump{\nabla u_h}\|_{L^{\infty}(e)}.
\end{align*}

We then deduce the estimate \eqref{eq:Eff2} using the estimates \eqref{eq:Eff1}, \eqref{eq:EffJump} and \eqref{eq:Ghbound}.\\


${\color{blue}\bullet}$ Now, we provide the proof of the estimate \eqref{eq:Eff4}. Let $T \in \cT_h$ be such that $T \cap \{\sigma_h <0\} \neq \emptyset$.
\par \noindent
 Using triangle inequality, we find
\begin{align*}
\|(u_h-\chi)^+\|_{L^{\infty}(T)} \leq \|u_h-u_h^{conf}\|_{L^{\infty}(T)} +\|(u_h^{conf}-\chi)^+\|_{L^{\infty}(T)}.
\end{align*}
\par \noindent
The first term in the right hand side of above estimate can be controlled by using \eqref{eq:apen3}. We estimate the second term as follows: In view of remark \ref{rmk:r1}, we have that either $T \in \mathbb{C}_h $ or $T \in \mathbb{M}_h $. If   $T \in \mathbb{C}_h $, then using the definition of $E_h(\cdot)$, we have $u_h^{conf}=\chi_h$ on $T$. Therefore,
\begin{equation} \label{eq:Me1}
\|(u_h^{conf}-\chi)^+\|_{L^{\infty}(T)}= \|(\chi_h-\chi)^+\|_{L^{\infty}(T)}.
\end{equation}

%
\par \noindent
Otherwise, if $T \in \mathbb{M}_h$,  then
\begin{align}\label{eq:Me2}
\|(u_h^{conf}-\chi)^+\|_{L^{\infty}(T)}\leq\|( u_h^{conf}-\chi_h)^+\|_{L^{\infty}(T)}+ \|(\chi_h-\chi)^+\|_{L^{\infty}(T)}.
\end{align}
Note that, since $T \in \mathbb{M}_h$, there exists a node $z \in \cV_T$ such that $u_h(z)=\chi_h(z)$  which implies $u_h^{conf}(z)=\chi_h(z)$.\\
\par \noindent
We claim that there exist an interior node $z \in \cV_T$ such that $u_h^{conf}(z)=\chi_h(z)$. This claim trivially holds if $T \subset \Omega$. If $T \cap \partial \Omega \neq \emptyset$, we have the following two possibilities.\\

 (i)  $T \cap \partial \Omega$ is a node say $T \cap \partial \Omega=\{p\}$. If $\chi_h(p)=\chi(p)<0$ then $0=u_h^{conf}(p)>\chi_h(p)$. Hence there exists an interior node $z \in \cV_T$ such that $u_h^{conf}(z)=\chi_h(z)$.\\

(ii)  $T \cap \partial \Omega$ is an edge/face say $T \cap \partial \Omega=e$. If $\chi(p)<0$ for all $p \in e$ then $u_h^{conf}>\chi_h$ on e since $u_h^{conf}=0 $ on $\partial \Omega$.  And hence there exists an interior node $z \in \cV_T$ such that $u_h^{conf}(z)=\chi_h(z)$.\\

\par \noindent
Thus with the assumption that $\chi_h<0$ on $T \cap \partial \Omega$, we have an interior node $p \in \cV_T$ such that $u_h^{conf}(p)=\chi_h(p)$. This allows us to use the quadratic growth property of a non-negative discrete function (see \cite[Section 2]{Baiocchi:1977:VI}, \cite[Lemma 6.4]{NSV:2002:Ptwise}), which yields
\begin{align}\label{eq:Me3}
\|(u_h^{conf}-\chi_h)^+\|_{L^{\infty}(T)}&\lesssim  \max_{e \in \cE_h: e \cap w_T \neq \emptyset} \left(\|h_e\sjump{\nabla u_h^{conf}}\|_{L^{\infty}(e)}+\|h_e\sjump{\nabla \chi_h}\|_{L^{\infty}(e)} \right) \notag\\
&\lesssim  \displaystyle{\max_{e \in \cE_h: e \cap w_T \neq \emptyset}} \left(\|h_e\sjump{\nabla (u_h^{conf}-u_h)}\|_{L^{\infty}(e)}+\|h_e\sjump{\nabla u_h}\|_{L^{\infty}(e)}+\|h_e\sjump{\nabla \chi_h}\|_{L^{\infty}(e)} \right)
\end{align}
Combining estimates \eqref{eq:Me1}, \eqref{eq:Me2}, \eqref{eq:Me3} and using \eqref{eq:Eff2}, \eqref{eq:EffJump} and \eqref{eq:apen3}, we deduce  \eqref{eq:Eff4}.

{
\begin{remark}
The definition of  $\sigma_h$ in this article is very local to each
simplex, thanks to the DG framework, which is computationally simpler. In
\cite{NSV:2002:Ptwise}, the definition of $\sigma_h$ at the boundary
vertices is modified (see \cite[page 169]{NSV:2002:Ptwise} ) by imposing  full
contact condition on the patches of those vertices, which in turn helps
them to prove the efficiency of the term
$\|(u_h-\chi)^+\|_{L^\infty(\{\sigma_h<0\})}$ whenever the discrete
free-boundary set intersects with the simplices that touch the boundary $\partial \Omega$ (see \cite[page 186]{NSV:2002:Ptwise}). Particularly, it is observed that the difficulty
arises when the obstacle $\chi=0$ on $\partial\Omega$ (given that the
solution $u=0$ on $\partial\Omega$ and the compatibility condition
$\chi\leq 0$ on $\partial\Omega$). In many research works the assumption
$\chi<0$ on $\partial\Omega$ is made, hence this condition used in this
article in proving the efficiency of the term
$\|(u_h-\chi)^+\|_{L^\infty(\{\sigma_h<0\})}$  is not restrictive. This also
avoids fine tuning the definition of $\sigma_h$ on the boundary nodes and
makes the computations simpler. Further, it enables us to prove an estimate
for $\sigma-\sigma_h$ in a dual norm, unlike in
\cite{NSV:2002:Ptwise} estimate is provided for $\sigma-\tilde\sigma_h$,
where  $\tilde\sigma_h$ is a modified Lagrange multiplier which is not
computable.

\end{remark}


}
\end{proof}

\section{Numerical Experiments}\label{sec:Numerics}
In this section, we present numerical results to
demonstrate the performance of a posteriori error estimator derived in Section \ref{sec:Reliaibility}. We have used the following algorithm for the adaptive refinement
\begin{equation*}
{\bf SOLVE}\longrightarrow  {\bf ESTIMATE} \longrightarrow {\bf
MARK}\longrightarrow {\bf REFINE}
\end{equation*}
In the SOLVE step, the primal-dual active set strategy
\cite{HK:2003:activeset} is used to solve the discrete obstacle problem \eqref{eq:FEM}. The algebraic setting of primal-dual active set method in accordance of the discrete problem \eqref{eq:FEM} is discussed in detail in  \cite{AGP:2021:PC}, and more numerical experiments can be found therein.
We compute the error estimator ${\eta}_h$ on each element $T \in \cT_h$ and
use maximum marking strategy  with
parameter $\theta=0.3$ to mark the elements for refinement.
Finally, we refine the mesh using the newest vertex bisection algorithm
and obtain a new adaptive mesh. For all the examples below, the penalty parameter $\gamma=25$.

\par
\noindent
We discuss numerical results for various test examples:\\
\par
\noindent

\textbf{Example1.} In this example, we consider $\Omega=(-1.5,1.5)^2$, $f=-2$,
$\chi=0$. The exact solution is given by:
\begin{equation*}
u := \left\{ \begin{array}{ll} r^2/2-\text{ln}(r)-1/2, & r\geq 1\\\\
0, & \text{otherwise},
\end{array}\right.
\end{equation*}
 where
$r^2=x^2+y^2$ for $(x,y)\in \R^2$.

Figure \ref{fig:SIPGer}  illustrate the behavior of the pointwise error $\|u-u_h\|_{L^{\infty}(\Omega)}$ and  a posteriori error estimator ${\eta}_h$ with respect to degrees of freedom (DOFs) for SIPG method. We observe that both error and estimator converge with the optimal rate DOFs$^{-1}$. This figure also depict the convergence behavior of individual estimators $\eta_i, i=1,2,3,4,6$. Note that, for this example $\eta_5$ is zero since $\chi=\chi_h=0$. Figure \ref{fig:MeshS}  represent efficiency indices with respect to degrees of freedom(leftmost sub-figure), adaptive mesh (center) and the discrete contact set (rightmost sub-figure) at refinement level 10.
In the contact region, the estimator should depend on the obstacle function $\chi$ while in the non-contact region it should be dictated by the load function $f$. Since the obstacle function is zero, we observe almost no mesh refinement in the contact zone.  \\

\begin{figure}[t]
\centering
\begin{subfigure}{.5\textwidth}
  \centering
  \includegraphics[width=7cm,height=6cm]{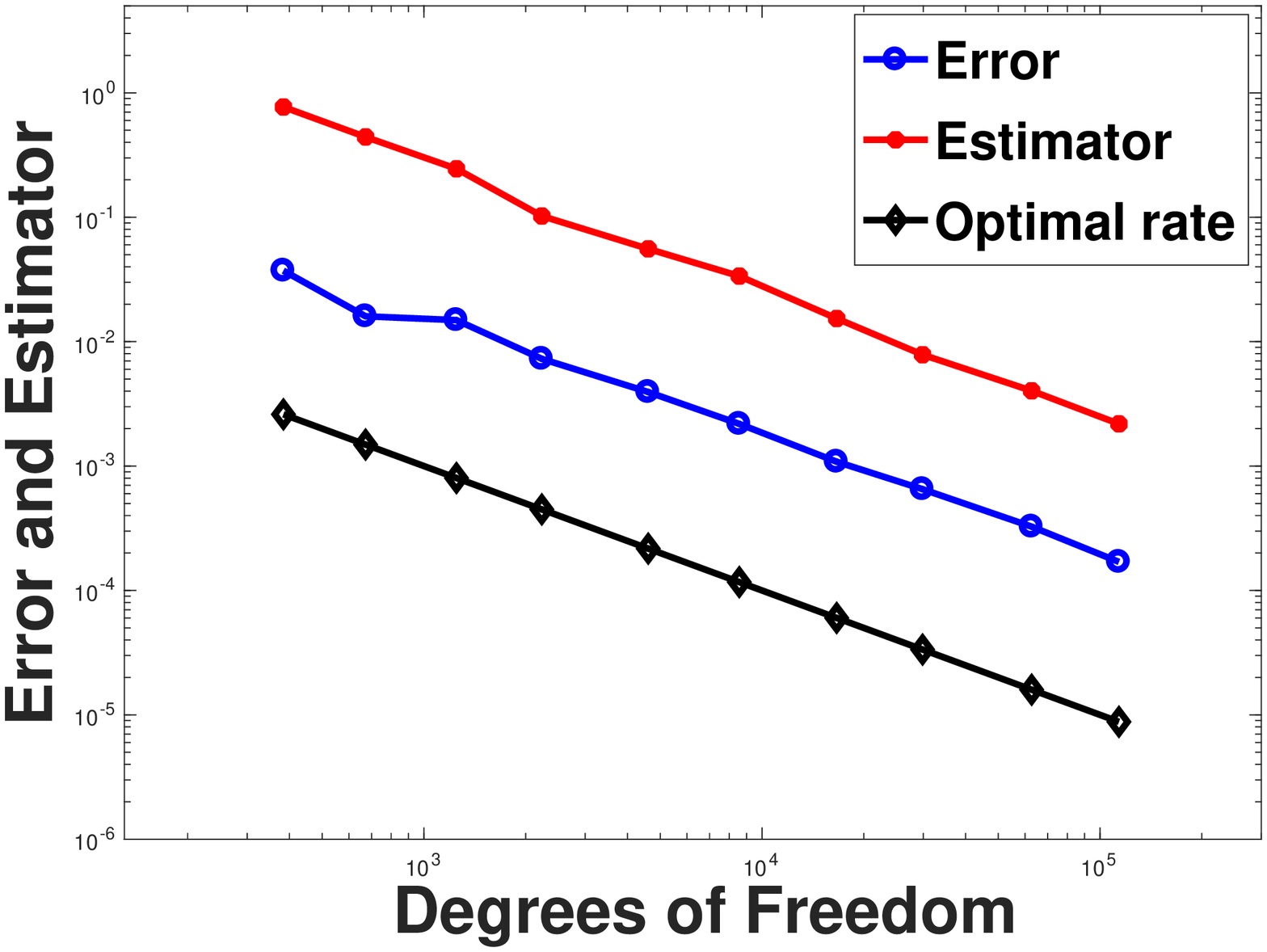}
\end{subfigure}%
\begin{subfigure}{.5\textwidth}
  \centering
  \includegraphics[width=7cm,height=6cm]{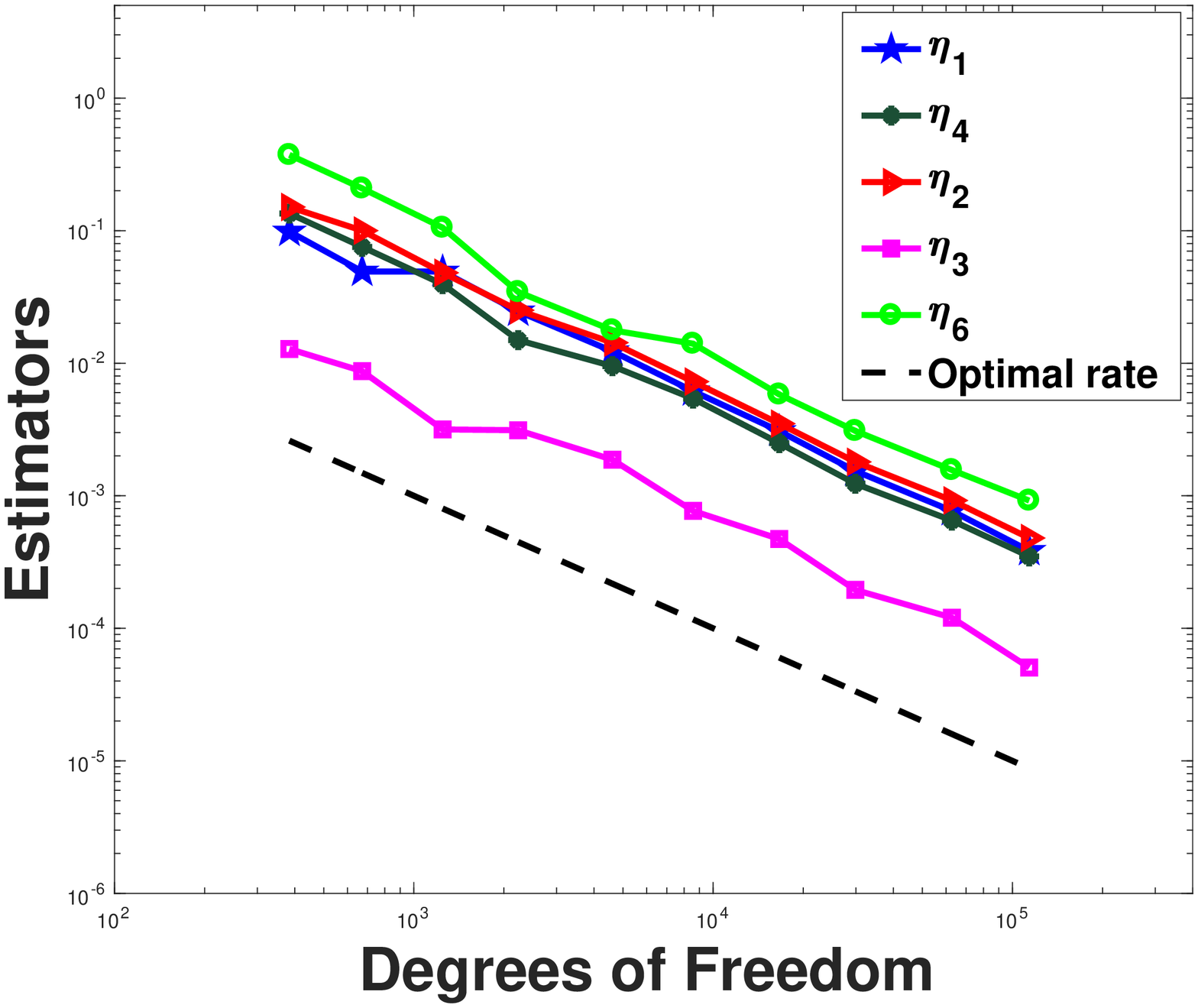}
\end{subfigure}
 \caption{{\large
Error and Estimator  for Example 1}} \label{fig:SIPGer}
\end{figure}
\begin{figure}[!ht]
\vskip -0.1cm
\centering
 \includegraphics[width=0.3\textwidth]{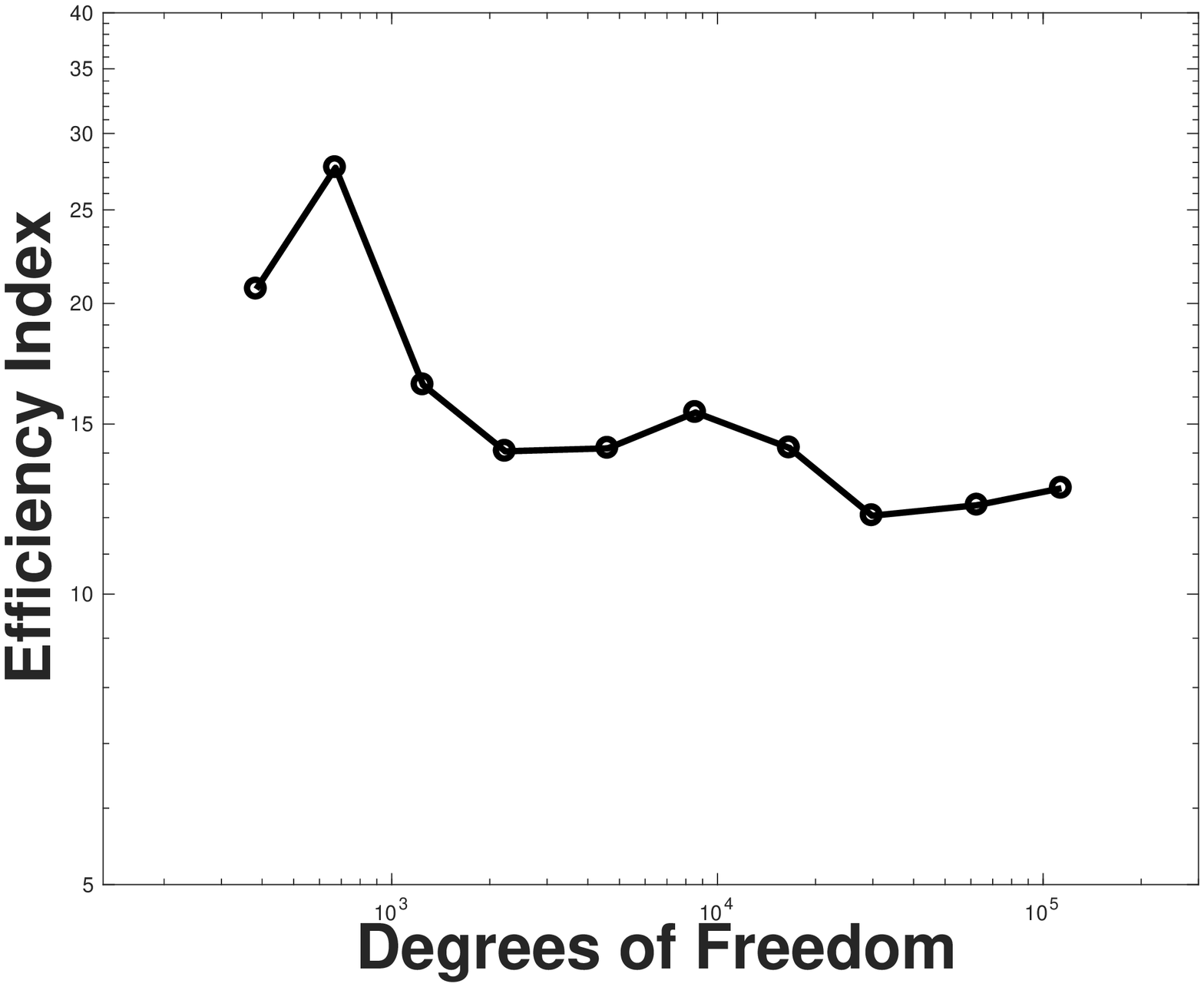}
 \hskip 0.04cm
 \includegraphics[width=0.3\textwidth]{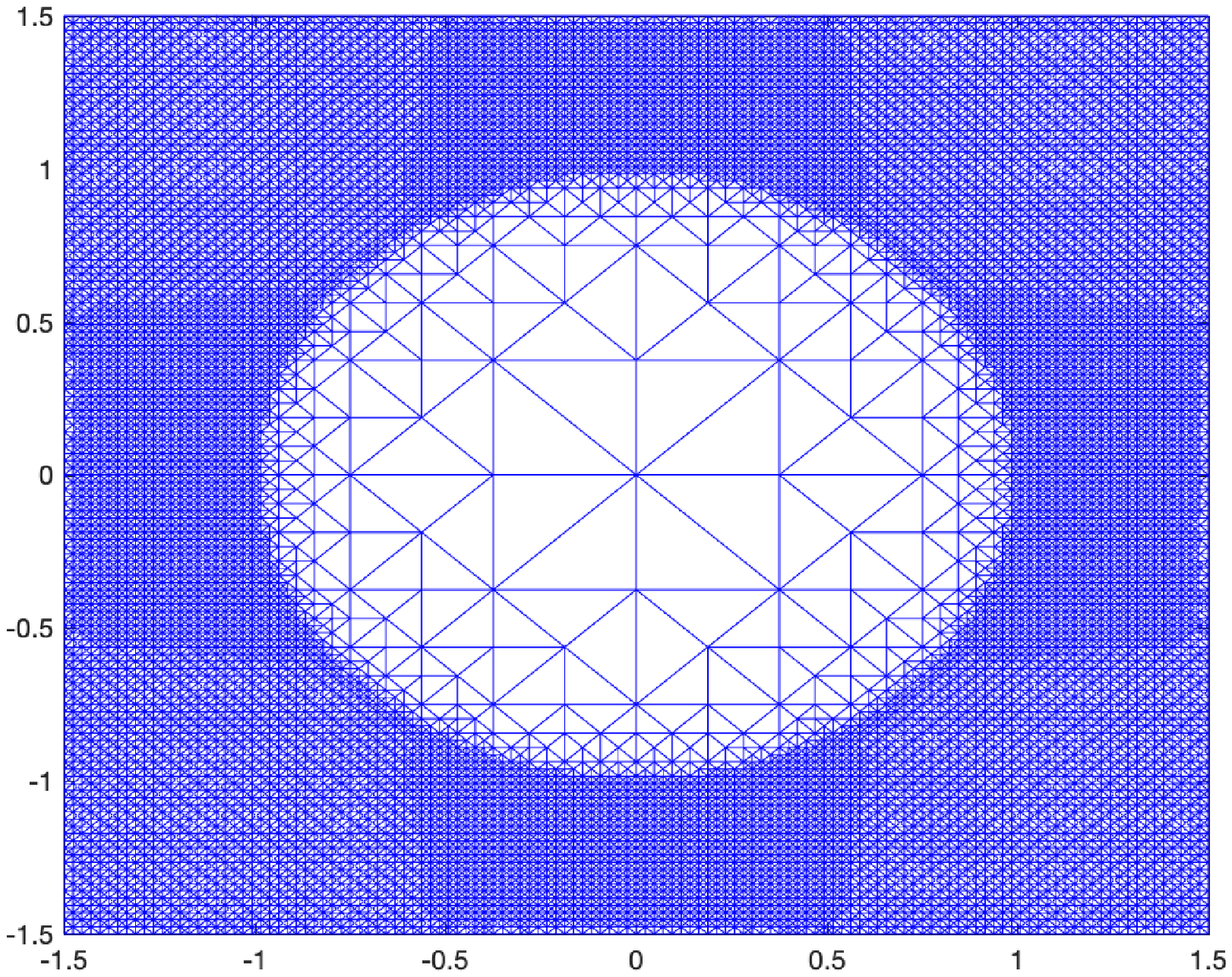}
  \hskip 0.04cm
 \includegraphics[width=0.3\textwidth]{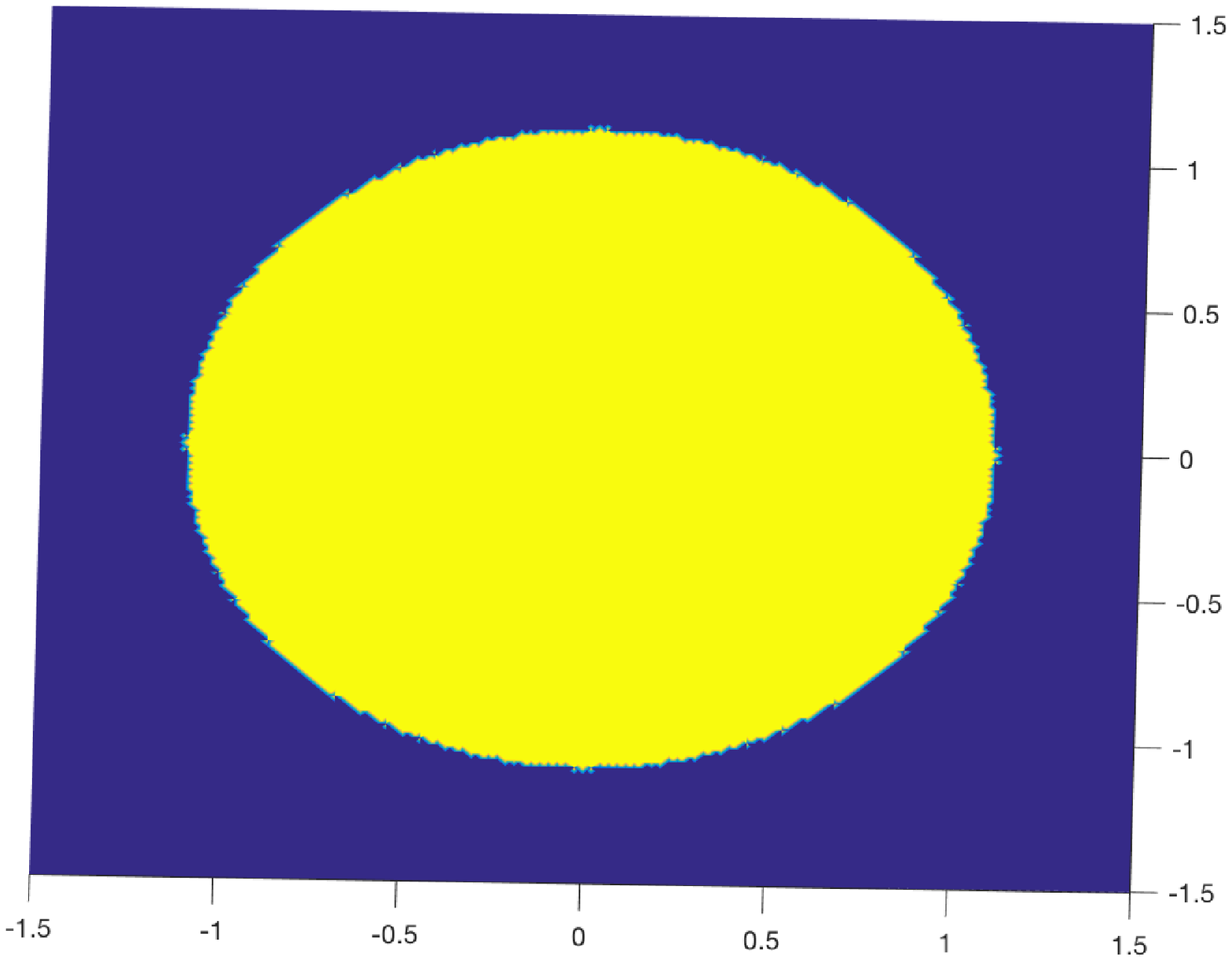}
 \vskip 0.1cm
  \caption{{\large
Efficiency Index, adaptive Mesh and discrete active set (yellow part)  for Example 1}} \label{fig:MeshS}
\end{figure}

%
\noindent
\textbf{Example 2.} To illustrate the efficacy of adaptive refinement, we consider this example with non-convex domain \cite{BCH:2007:VI}. Therein, we have the following data
\begin{align*}
\Omega &= (-2,2)^2 \setminus [ 0, 2) \times (-2, 0], \quad \chi=0, \\
u &= r^{2/3} sin(2\theta/3 ) \gamma_1(r),\\
f&=-r^{2/3} sin(2 \theta/3) \Big(\frac{\gamma_1^{'}(r)}{r} + \gamma_1^{''}(r) \Big)-\frac{4}{3}r^{-1/3}sin(2 \theta /3) \gamma_1^{'}(r)-\gamma_2(r)\\
\text{where}\\
\gamma_1(r) &= \begin{cases}  1, \quad  \tilde{r}<0 \\
-6 \tilde{r}^5 + 15 \tilde{r}^4 -10 \tilde{r}^3 +1 , \quad  0 \leq \tilde{r} <1 \\
                    0  , \quad \tilde{r} \geq 1,\\
\end{cases}\\
\gamma_2(r) &= \begin{cases} 0, \quad r \leq \frac{5}{4} \\
  1, \quad \mbox{otherwise}
 \end{cases}\\
 \text{with}\quad
 \tilde{r}= 2(r-1/4).
\end{align*}

The behavior of the true error and the error estimator for SIPG  method is depicted in Figure \ref{fig:SIPGerEx2}. 
This figure ensures the optimal convergence (rate DOFs$^{-1}$  ) of the error and the estimator together with the reliability of the estimator. The convergence behavior of single estimators $\eta_i, i=1,2,3,4,6$ in maximum norm is also depicted in \ref{fig:SIPGerEx2}. The efficiency of the estimator is shown in Figure \ref{fig:MeshSEx2} (leftmost sub-figure). The adaptive mesh refinement and the discrete contact set at refinement level 26 are also reported in Figure \ref{fig:MeshSEx2}. From these figures we observe that the estimator captures the singular behavior of the solution very well. The mesh refinement near the free boundary is higher because of the large jump in gradients. \\

\begin{figure}
\centering
\begin{subfigure}{.5\textwidth}
  \centering
  \includegraphics[width=6.5cm,height=6cm]{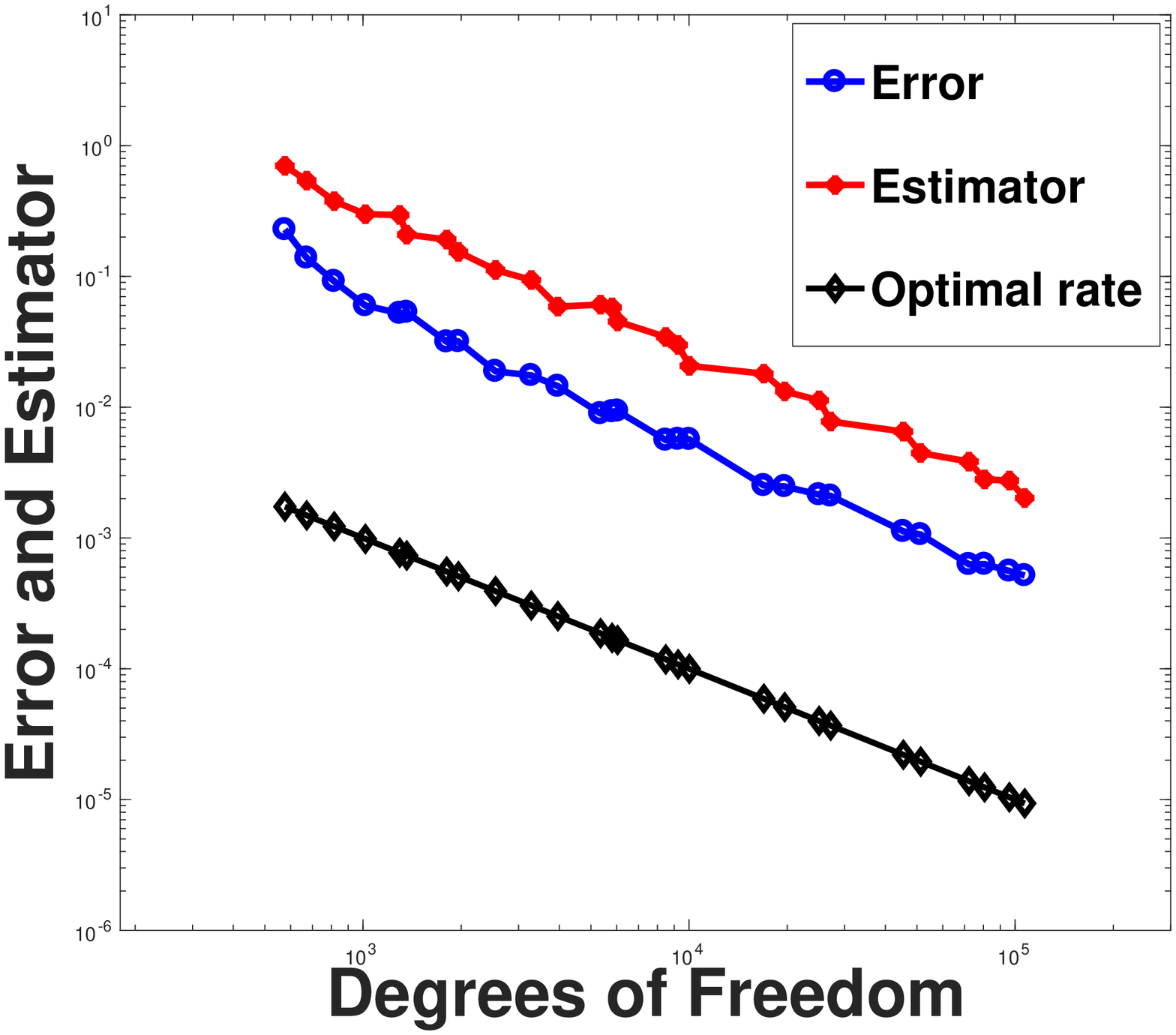}
\end{subfigure}%
\begin{subfigure}{.5\textwidth}
  \centering
  \includegraphics[width=7cm,height=6cm]{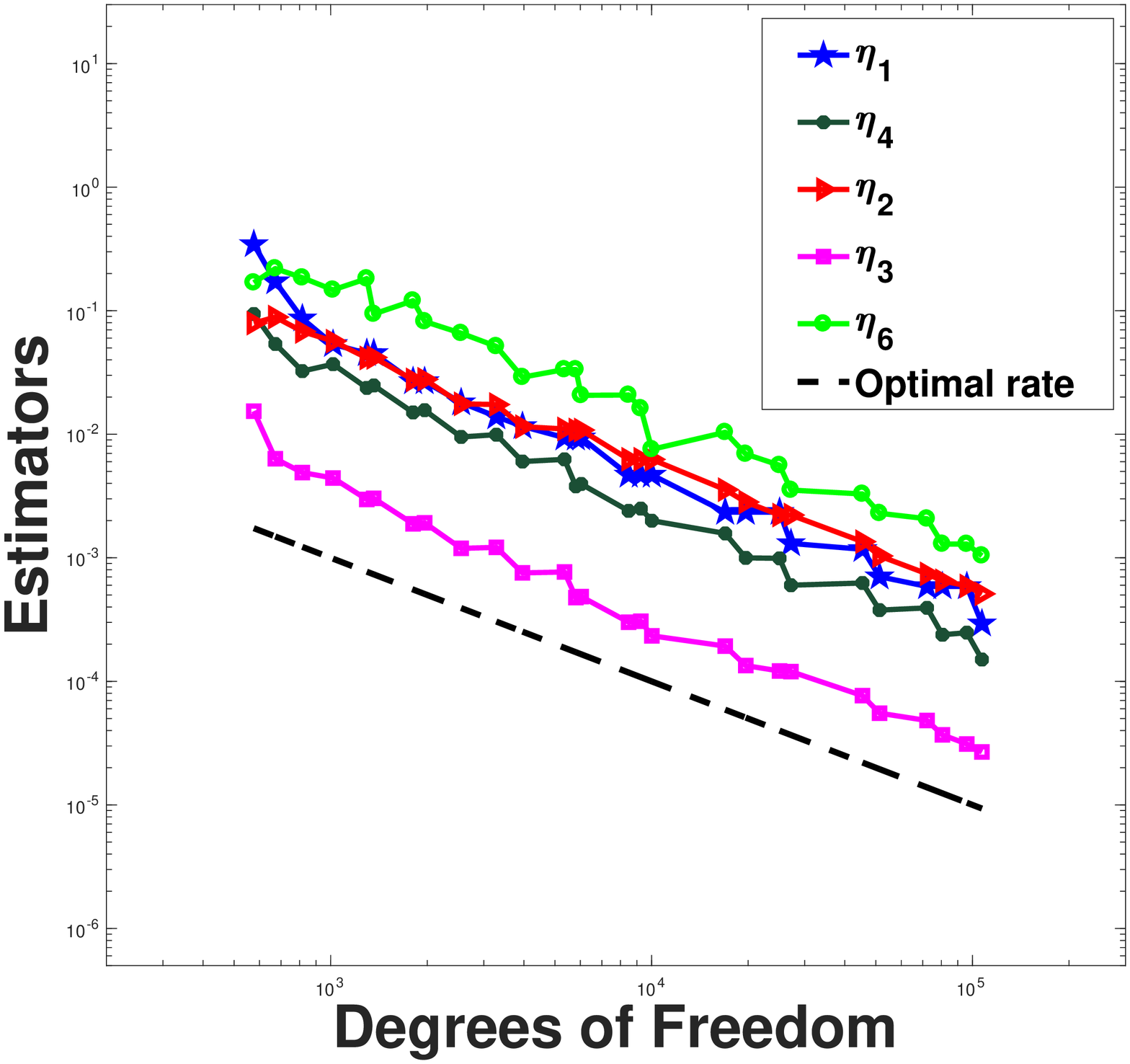}
\end{subfigure}
\caption{\large
Error and Estimator  for Example 2}
\label{fig:SIPGerEx2}
\end{figure}


\begin{figure}[!ht]
\vskip -0.1cm
\centering
 \includegraphics[width=0.3\textwidth]{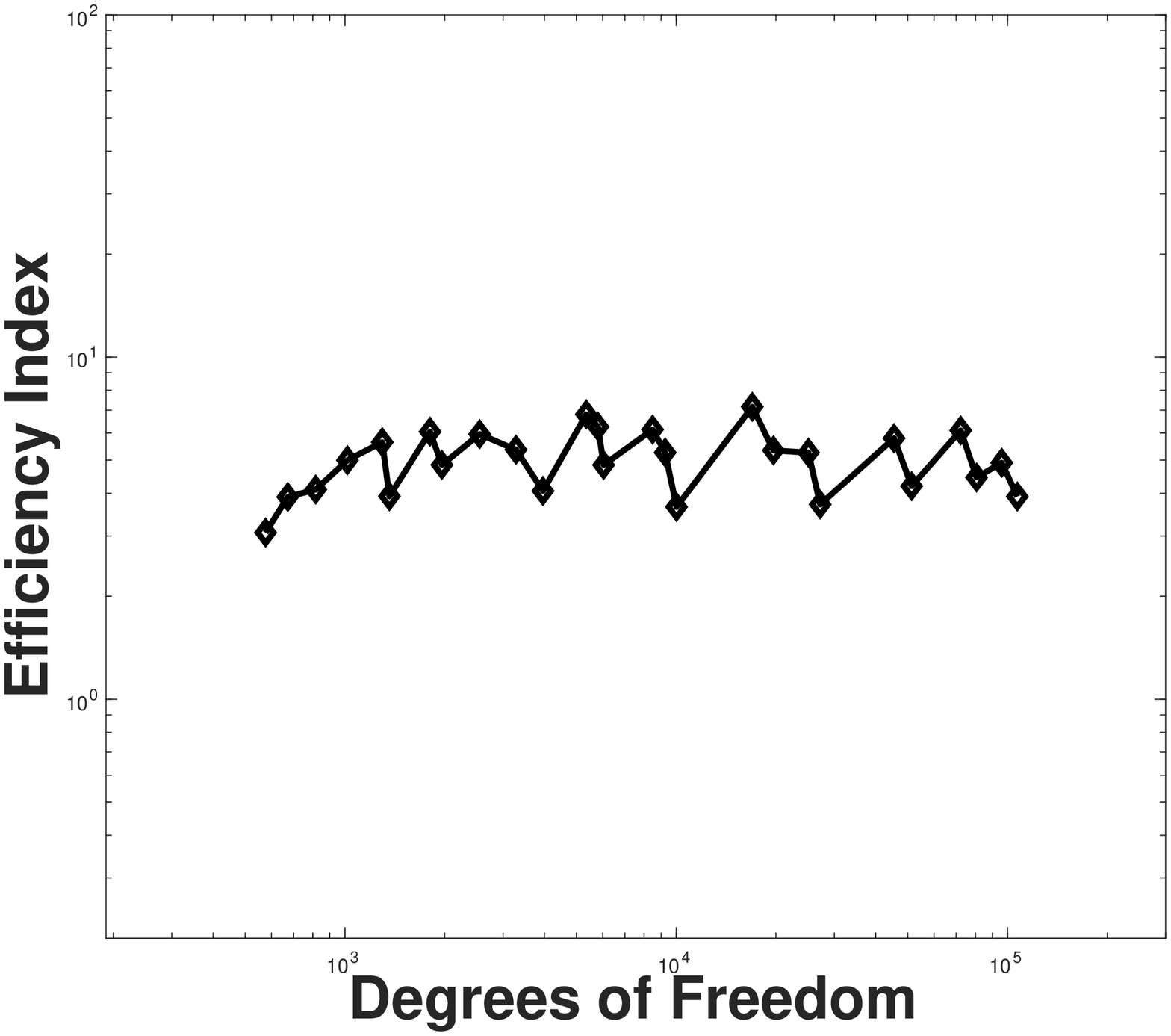}
 \hskip 0.04cm
 \includegraphics[width=0.3\textwidth]{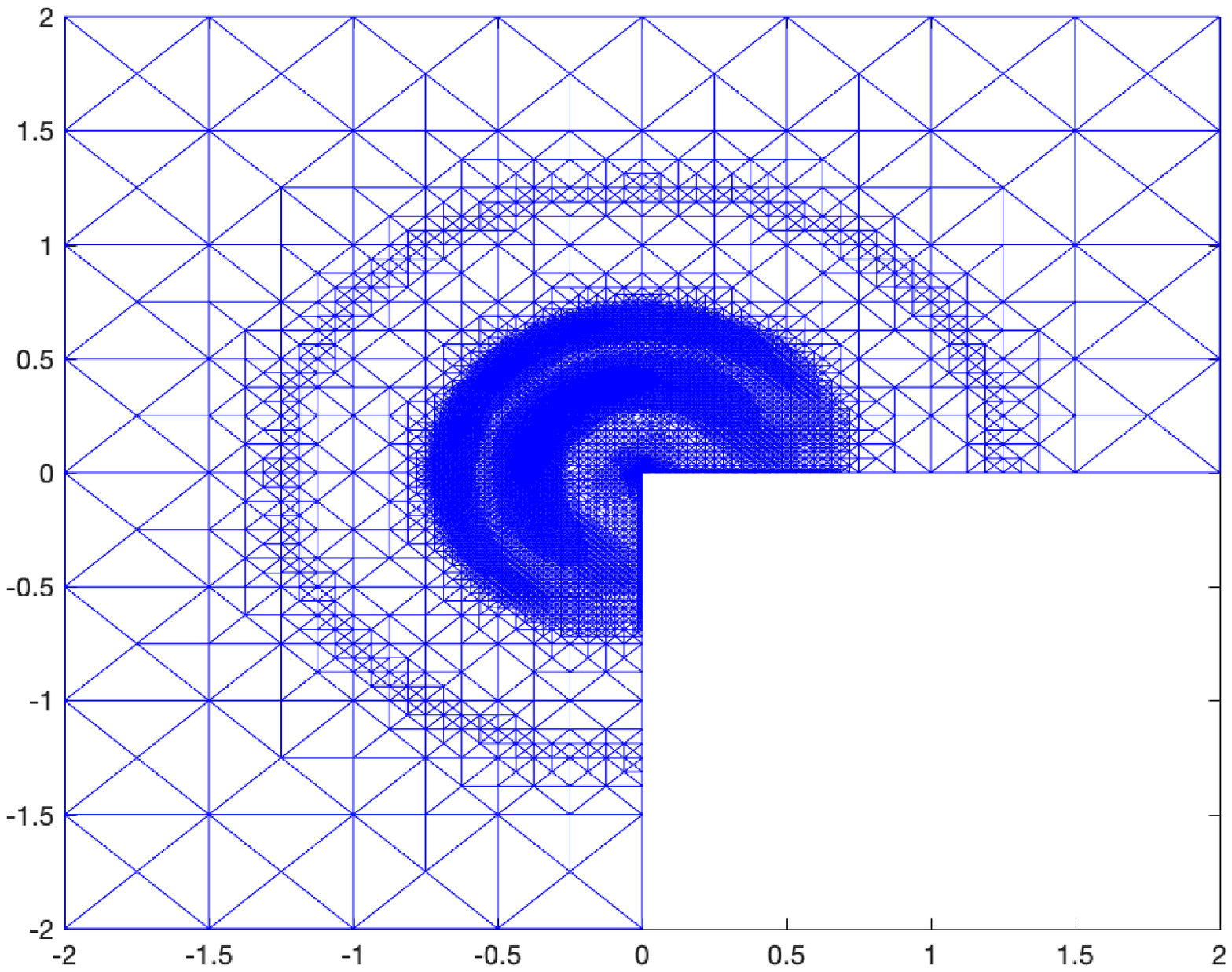}
  \hskip 0.04cm
 \includegraphics[width=0.3\textwidth]{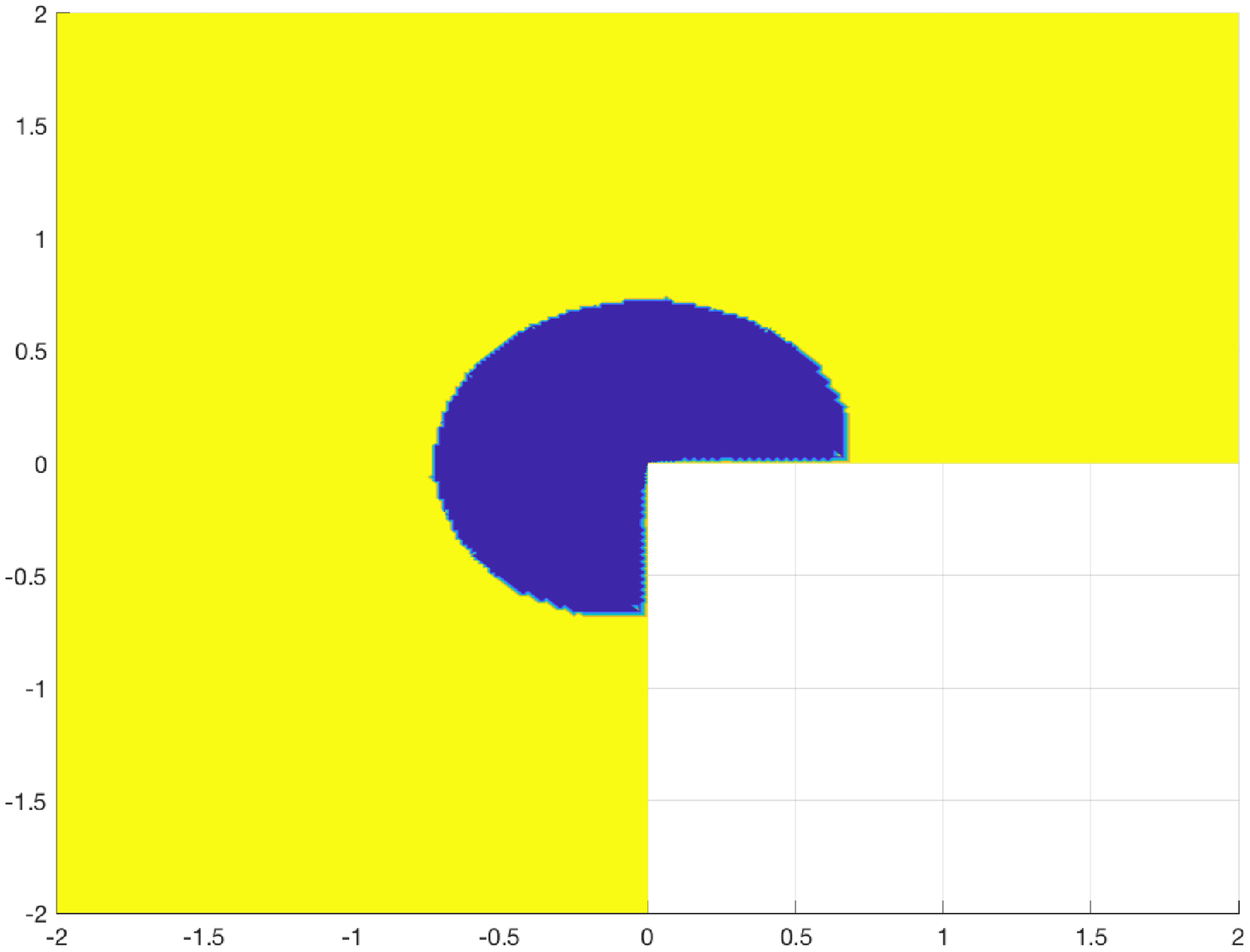}
 \vskip 0.1cm
  \caption{\large
Efficiency Index, adaptive Mesh and discrete active set (yellow part)  for Example 2}
\label{fig:MeshSEx2}
\end{figure}


\vspace{1cm}
\textbf{Example 3} In this example, we consider the following data from \cite{NSV:FullyLoc}:
\begin{align*}
\Omega &= (-2,2)\times (-1,1) ,\quad r^2=x^2+y^2,  \\
\chi &=10-6(x^2-1)^2-20 (r^2-x^2).\\
\end{align*}
The exact solution is not known for this example. Figure \ref{fig:EstmSEx6} illustrate the behavior of the error estimator $\eta_h$ together with the individual estimators $\eta_i, i=1,...6$ for the load function $f=0$ and $f=-15$. We clearly observe that in both cases the estimator converges with optimal rate (DOFs$^{-1}$). The adaptive mesh and discrete contact set at  refinement level 13 for $f=0$ and $f=-15$ are shown in figures  \ref{fig:MeshSEx6} and \ref{fig:ActvSEx6}, respectively. The graph of the obstacle can be viewed as two hills connected by a saddle. As the load $f$ increases  the change in the contact region can be observed from Figure \ref{fig:ActvSEx6} and as expected, we observe more refinement near the free boundaries.\\

\begin{figure}
\centering
\begin{subfigure}{.5\textwidth}
  \centering
  \includegraphics[width=7cm,height=6cm]{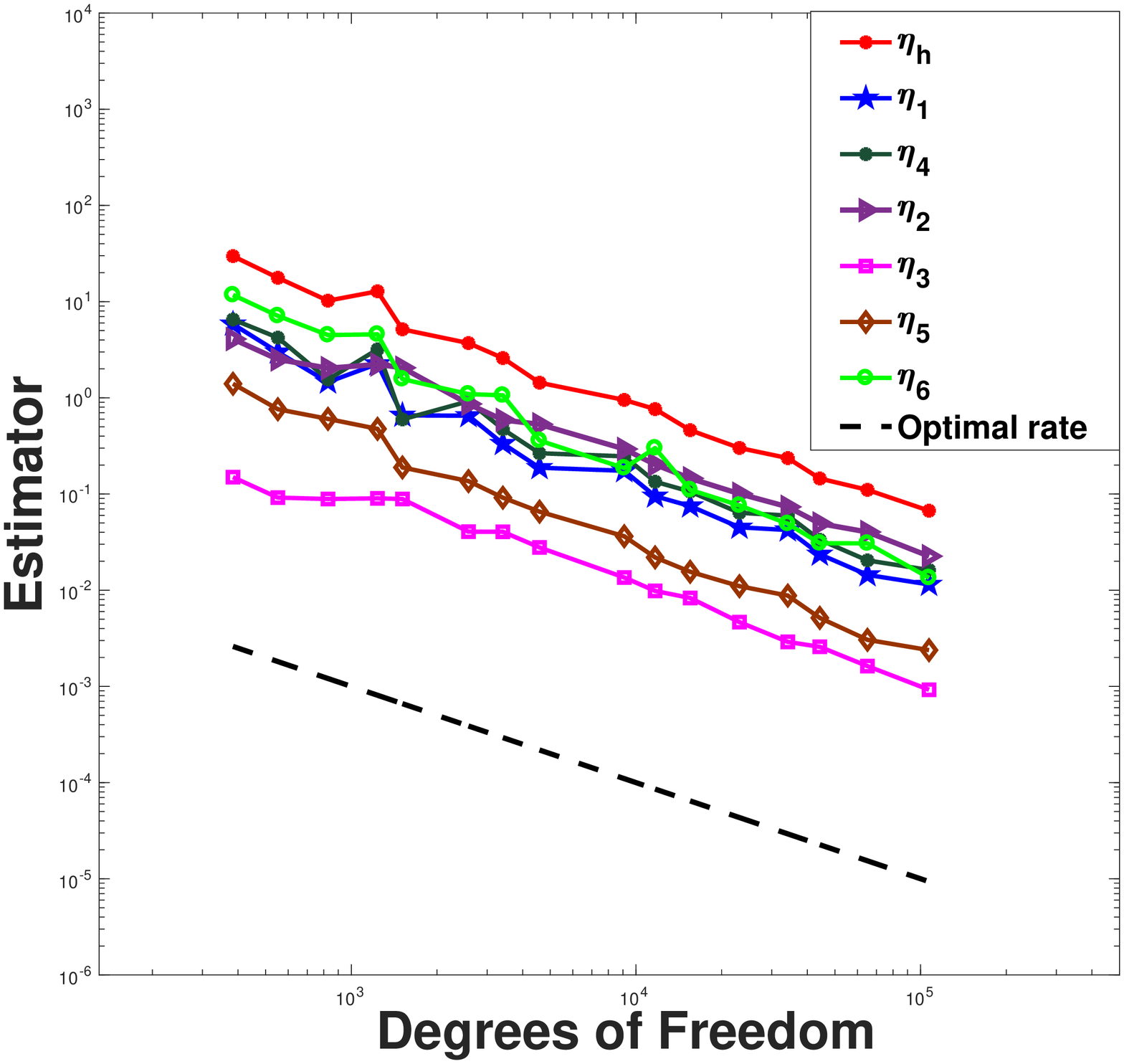}
\end{subfigure}%
\begin{subfigure}{.5\textwidth}
  \centering
    \includegraphics[width=7cm,height=6cm]{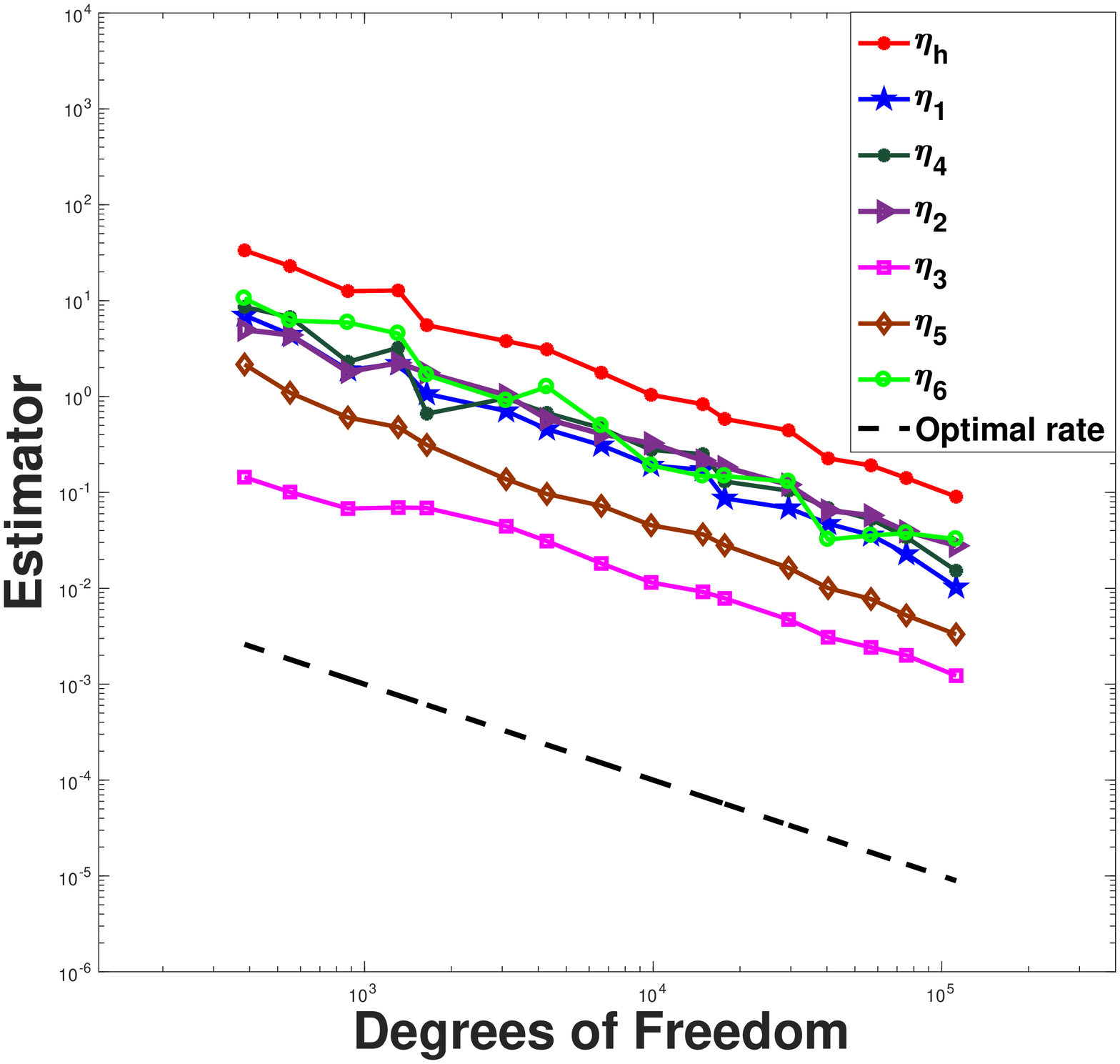}
\end{subfigure}
\caption{
Estimator  for Example 3 for $f=0$ and $f=-15$}
\label{fig:EstmSEx6}
\end{figure}

\begin{figure}
\centering
\begin{subfigure}{.5\textwidth}
  \centering
  \includegraphics[width=7cm,height=6cm]{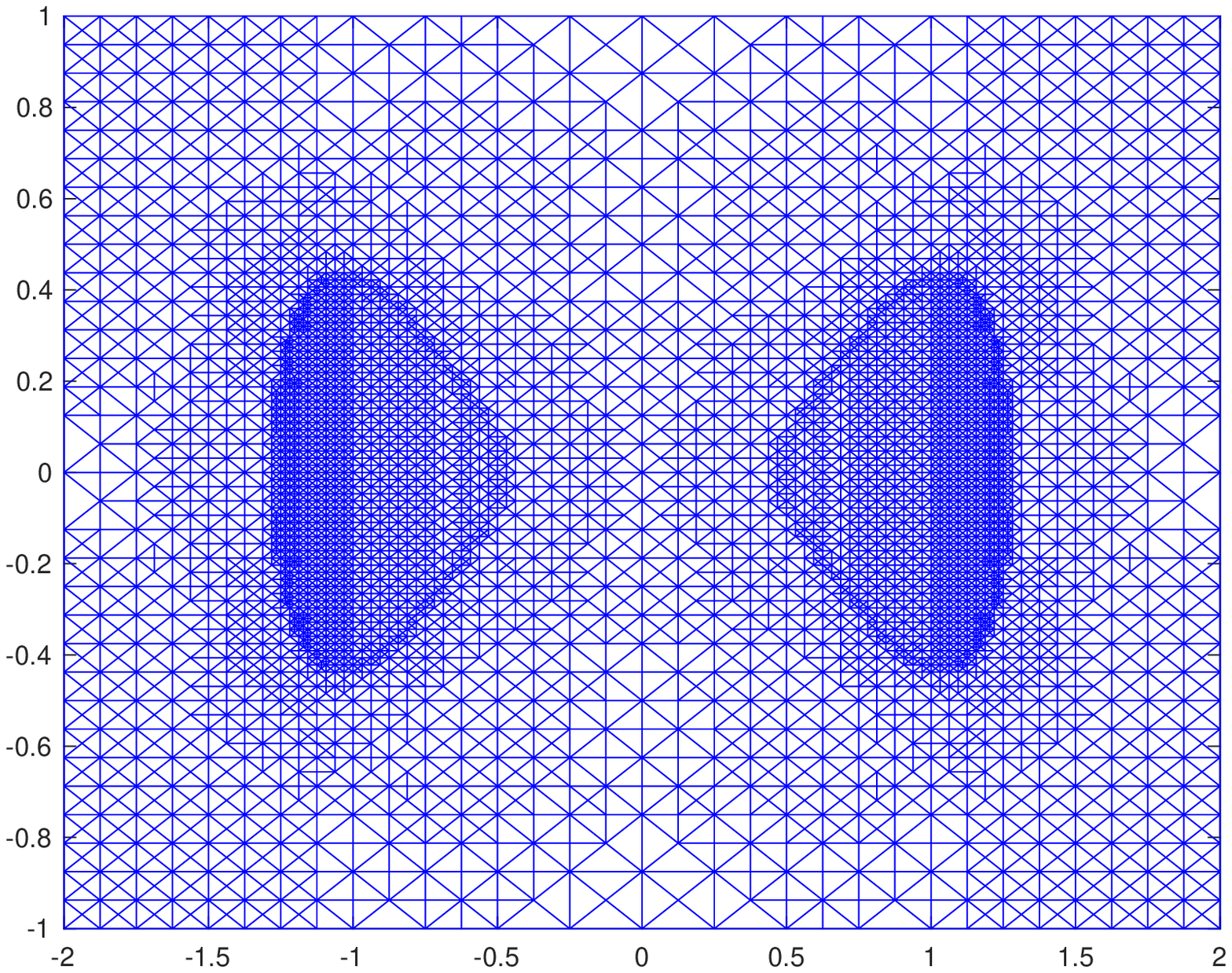}
\end{subfigure}%
\begin{subfigure}{.5\textwidth}
  \centering
    \includegraphics[width=7cm,height=6cm]{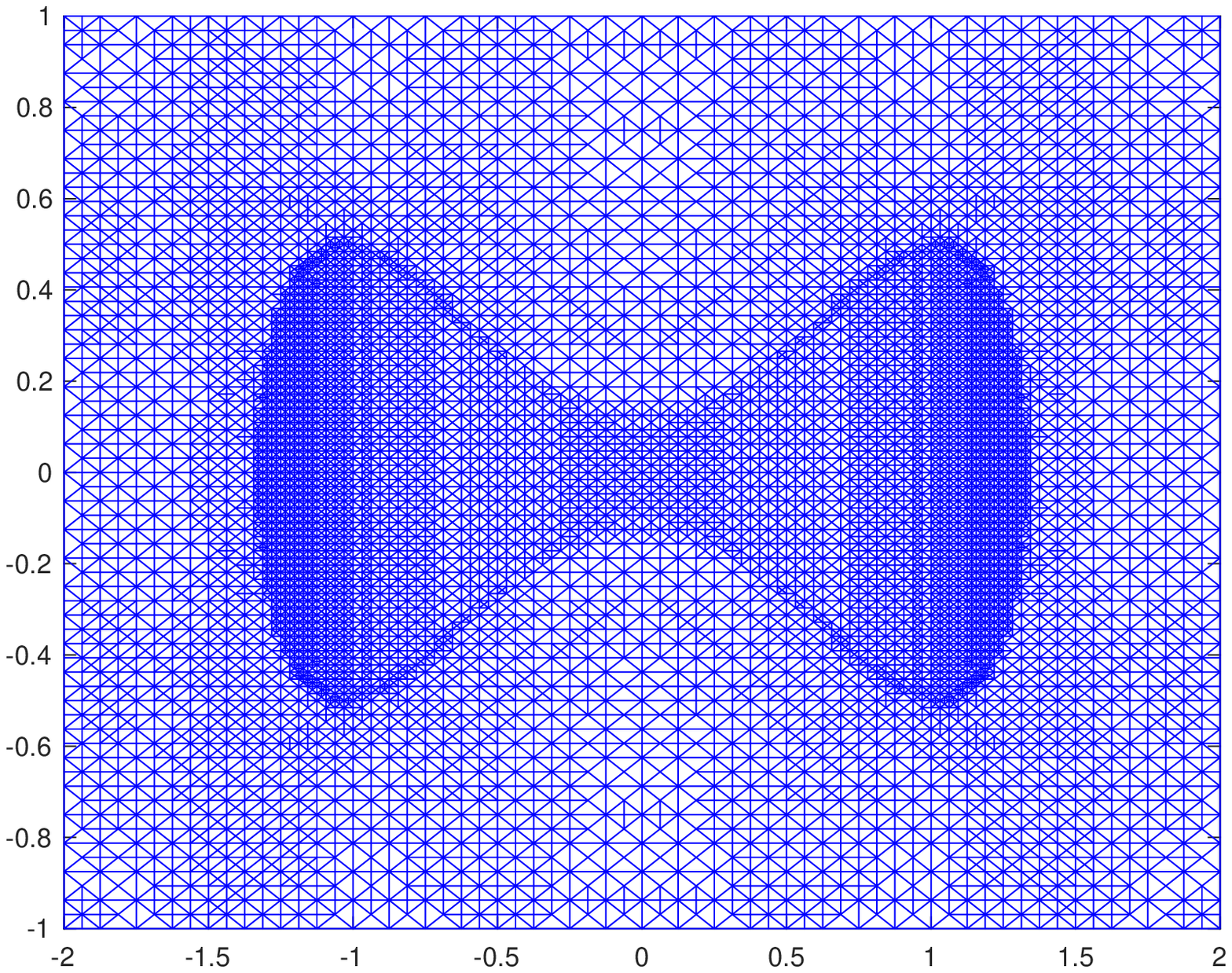}
\end{subfigure}
\caption{
Adaptive mesh for Example 3 for $f=0$ and $f=-15$}
\label{fig:MeshSEx6}
\end{figure}

\begin{figure}
\centering
\begin{subfigure}{.5\textwidth}
  \centering
  
   \includegraphics[width=7cm,height=6cm]{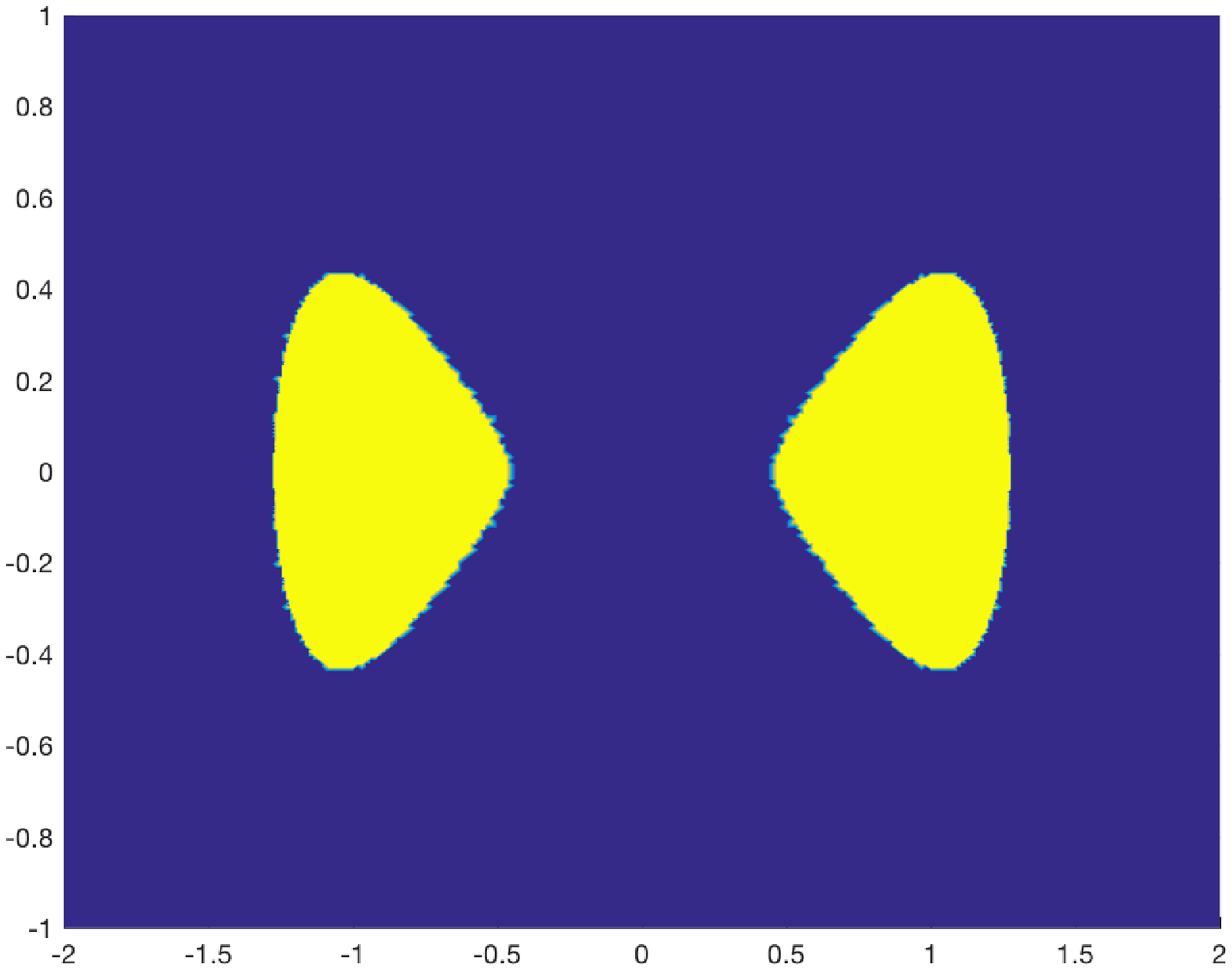}
\end{subfigure}%
\begin{subfigure}{.5\textwidth}
  \centering
  \includegraphics[width=7cm,height=6cm]{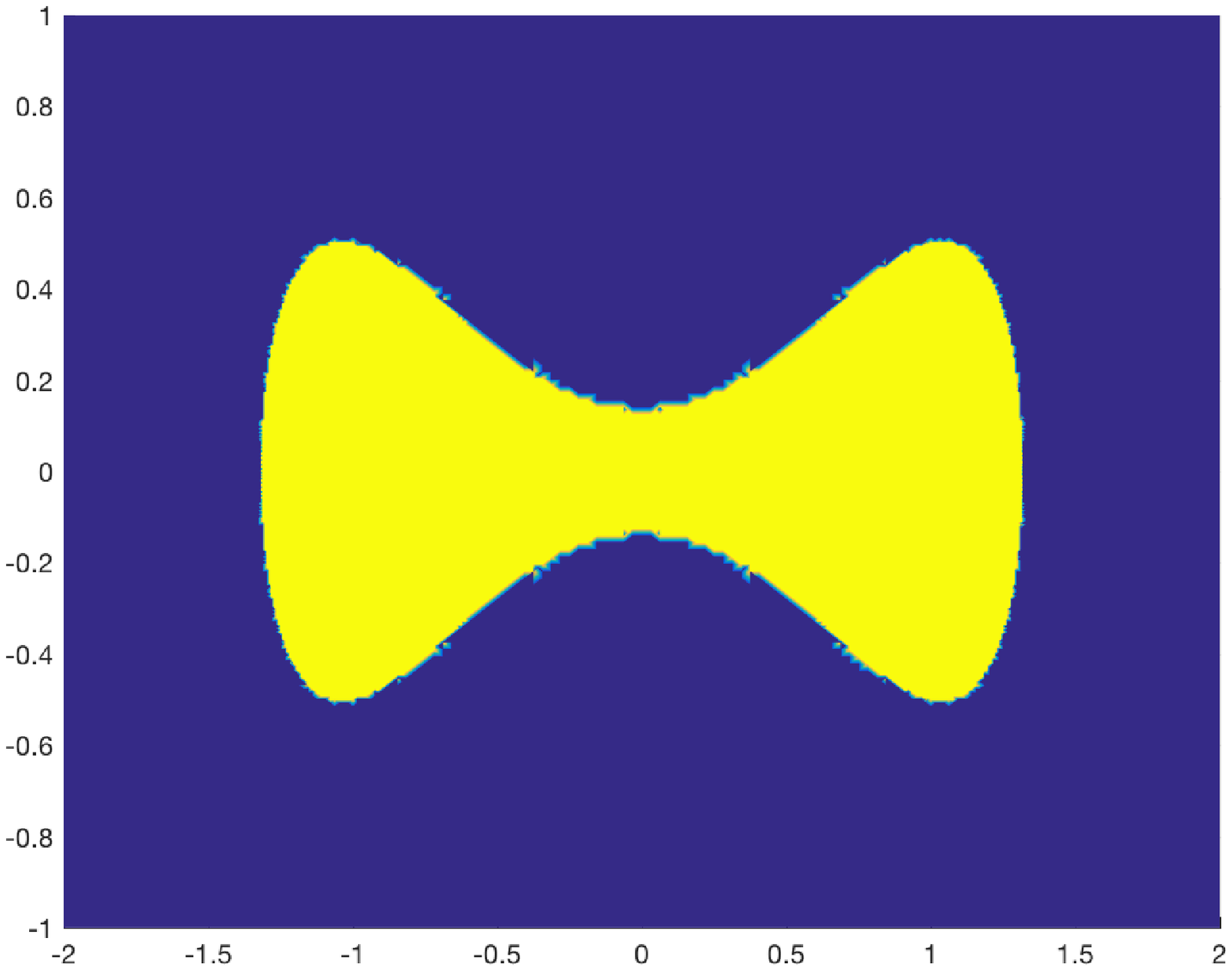}
\end{subfigure}
\caption{
 Discrete active set (yellow part)  for Example 3 for $f=0$ and $f=-15$}
\label{fig:ActvSEx6}
\end{figure}

%

\bibliographystyle{amsplain}
\bibliography{Obstacle1}

\end{document}